\documentclass{amsart}

\usepackage[colorlinks=true, linkcolor=blue, citecolor=blue]{hyperref}
\usepackage{stmaryrd}

\usepackage{amsmath,amsthm,amssymb}
\usepackage[margin=2.5cm]{geometry}
\usepackage{enumitem}
\usepackage{tikz-cd}
\usepackage{mathrsfs}

\DeclareFontFamily{U}{wncy}{}
\DeclareFontShape{U}{wncy}{m}{n}{<->wncyr10}{}
\DeclareSymbolFont{mcy}{U}{wncy}{m}{n}
\DeclareMathSymbol{\Sha}{\mathord}{mcy}{"58}

\setlist[itemize]{label=--}

\newtheorem{theorem}{Theorem}[section]
\newtheorem{lemma}[theorem]{Lemma}
\newtheorem{prop}[theorem]{Proposition}
\newtheorem{cor}[theorem]{Corollary}
\newtheorem{conj}[theorem]{Conjecture}
\theoremstyle{definition}
\newtheorem{definition}[theorem]{Definition}
\newtheorem{example}[theorem]{Example}

\newtheorem*{def*}{Definition}
\newtheorem*{ex*}{Example}
\theoremstyle{remark}
\newtheorem{remark}[theorem]{Remark}
\newtheorem*{remark*}{Remark}

\newcommand{\A}{\mathbb{A}}

\newcommand{\bra}{\langle}
\newcommand{\C}{\mathbb{C}}
\newcommand{\CH}{\mathrm{CH}}

\newcommand{\cont}{\mathrm{cont}}

\newcommand{\F}{\mathbb{F}}

\newcommand{\Fr}{\mathrm{Frob}}
\newcommand{\G}{\mathbf{G}}
\newcommand{\GL}{\mathrm{GL}}

\newcommand{\h}{\mathrm{H}}

\newcommand{\iso}{\xrightarrow{\,\sim\,}}

\newcommand{\ket}{\rangle}
\newcommand{\m}{\mathfrak{m}}
\newcommand{\mb}[1]{\mathbf{#1}}

\newcommand{\Or}{\mathcal{O}}

\newcommand{\Q}{\mathbb{Q}}
\newcommand{\R}{\mathbb{R}}

\newcommand{\sh}[1]{\mathcal{#1}}

\newcommand{\U}{\mathrm{U}}

\newcommand{\Z}{\mathbb{Z}}

\renewcommand{\Re}{\operatorname{Re}}

\DeclareMathOperator{\diag}{diag}

\DeclareMathOperator{\Gal}{Gal}
\DeclareMathOperator{\Hom}{Hom}
\DeclareMathOperator{\Ind}{Ind}

\DeclareMathOperator{\Res}{Res}

\DeclareMathOperator{\spec}{Spec}

\DeclareMathOperator{\Tr}{Tr}

\theoremstyle{definition}
\newtheorem{assumption}[theorem]{Assumption}
\newtheorem{observation}[theorem]{Observation}

\renewcommand{\G}{\mathbb{G}}
\newcommand{\bH}{\mathbf{H}}
\newcommand{\bG}{\mathbf{G}}
\newcommand{\Sh}{\mathrm{Sh}}

\newcommand{\BC}{\mathrm{BC}}
\newcommand{\CO}{{\mathcal {O}}}
 \newcommand{\CS}{{\mathcal {S}}}
\newcommand{\BZ}{{\mathbb {Z}}}

\newcommand{\wt}{\widetilde}

\newcommand{\indf}{\mathbf{1}}

\newcommand{\tX}{{\widetilde{X}}}

\newcommand{\open}[1]{\mathring{#1}}
\newcommand{\shear}{{\mathbin{\mkern-6mu\fatslash}}}

\begin{document}

\title{Euler systems and relative Satake isomorphism}
\author{Li Cai} 
\address[L.C.]{Academy for Multidisciplinary Studies, Beijing National Center for Applied Mathematics, Capital Normal University, Beijing, 100048, People's Republic of China
} 
\email{caili@cnu.edu.cn}

\author{Yangyu Fan} 
\address[Y.F.]{Key Laboratory of Algebraic Lie Theory and Analysis of Ministry of Education, School of Mathematics and Statistics, Beijing Institute of Technology, Beijing, 100081, People's Republic of China
} 
\email{yangyu.fan@bit.edu.cn}

\author{Shilin Lai} 
\address[S.L.]{Department of Mathematics, University of Michigan, 530 Church Street, Ann Arbor, MI 48109, USA} 
\email{shilinl@umich.edu}

\begin{abstract}
  We explain how the unramified Plancherel formula in the relative Langlands program gives a natural way of constructing test vectors which satisfy the tame norm relations of an Euler system. This uniformly recovers many of the known Euler systems, and in the twisted Friedberg--Jacquet setting, we produce a new split anticyclotomic Euler system.
\end{abstract}

\maketitle
\setcounter{tocdepth}{1}
\tableofcontents

\section{Introduction}
The Bloch--Kato conjecture suggests a deep connection between the arithmetic of motives and special values of $L$-functions. The method of Euler systems is one way of understanding this connection. An Euler system consists of a family of motivic classes satisfying two types of relations: ``wild'' ones for relations in a $p$-adic tower, and ``tame'' ones for relations at places away from $p$. Experiences suggest that the wild norm relations require quite special conditions on $p$, such as some form of ordinarity. On the other hand, tame norm relations seem to exist in much greater generalities, and they are already enough to deduce some cases of the Bloch--Kato conjecture, cf.~\cite{JNS, LaiSkinner}. However, in the many known examples of Euler systems, the construction of classes satisfying the tame norm relations usually requires some \emph{ad hoc} choices followed by extensive case-by-case calculations.

On the automorphic side of this picture, one way to get a handle on special values of $L$-functions is through period integrals. Correspondingly, all known constructions of Euler systems are based (perhaps implicitly) on motivic interpretations of period integrals. Recently, Ben-Zvi--Sakellaridis--Venkatesh proposed a relative Langlands program, which is a far-reaching framework for organizing these period integrals centered around spherical varieties and their generalizations \cite{BZSV}. It is natural to ask if there is an arithmetic analogue of this framework, which should organize the many constructions of motivic classes (or their realizations) in the literature.

In this paper, we examine Euler systems from the relative Langlands point of view. We explain that this gives an automatic way of producing tame norm relations. Our method is computation-free, and we uniformly recover many of the known examples, some of which are summarized in Table~\ref{tab:Examples}. More details can be found in \S\ref{ss:Diagonal}, \ref{ss:Eisenstein}.

\begin{table}[ht]
    \centering
    \begin{tabular}{l|l|l}
    Group & Spherical variety & Attribution\\ \hline
    $\U(n)\times\U(n+1)$ & $\U(n)\backslash\U(n)\times\U(n+1)$ & Lai--Skinner \cite{LaiSkinner}\\ \hline
    $\U(2n)$ & $\U(n)\times\U(n)\backslash\U(2n)$ &  Graham--Shah \cite{GrahamShah}\\ \hline
    \multicolumn{2}{c|}{Inner form of $\uparrow$} & Twisted Friedberg--Jacquet\\ \hline
    \hline
    $\GL_2\times\GL_2$ & $\bG\times^{\GL_2}\mathtt{std}$ &  Lei--Loeffler--Zerbes \cite{LLZ}\\ \hline
    $\Res_{F/\Q}\GL_2$ & $\uparrow$ &  Grossi \cite{GrossiAF}\\ \hline
    $\mathrm{GU}(2,1)$ & $\uparrow$ & Loeffler--Skinner--Zerbes \cite{LSZ-U3}\\ \hline
    $\mathrm{GSp}_4\times_{\G_m}\GL_2$ & $\bG\times^{\GL_2\times_{\G_m}\GL_2}\mathtt{std}$ & Hsu--Jin--Sakamoto \cite{HJS}
    \end{tabular}
    \caption{Examples of spherical varieties giving rise to tame parts of Euler systems}
    \label{tab:Examples}
\end{table}

\subsection{Simple case of main result}
We will state the application of our results for the pushforward of cycles, covering the first three rows of Table~\ref{tab:Examples}. We construct split anticyclotomic Euler systems in the sense of Jetchev--Nekov\'a\v{r}--Skinner \cite{JNS}, which we will simply call JNS Euler systems.

Let $E/F$ be a CM extension. Suppose $\mb{H}\hookrightarrow\mb{G}$ are reductive groups with Shimura varieties $\Sh_{\mb{H}}\hookrightarrow\Sh_{\bG}$ defined over $E$ such that the basic numerology
\[
    \dim\Sh_\bG=2\dim\Sh_{\mb H}+1
\]
holds. After fixing level structures for $\mb{G}$ and $\mb{H}$, this defines a special cycle on $\Sh_\bG$ in the arithmetic middle dimension, which is expected to be related to the central derivative of an $L$-function.

We set up some notations for a JNS Euler system. Let $\mathscr{L}$ be the set of places of $F$ which split in $E$, with an explicit finite set of exceptions. For each $\ell\in\mathscr{L}$, fix a place $\lambda$ of $E$ above $\ell$, and let $\mathrm{Frob}_\lambda$ be the arithmetic Frobenius at $\lambda$. Let $\mathscr{R}$ be the set of square-free products of places in $\mathscr{L}$. For $\m\in\mathscr{R}$, let $E[\m]$ be the ring class field of conductor $\m$, so it is associated to the order $\Or_F+\m\Or_E$ by class field theory. Our first result is the construction of a motivic Euler system.

\begin{theorem}[Proposition~\ref{prop:cycle}+Corollary~\ref{cor:Abstract}]\label{thm:1}
    Let $d=\dim\Sh_{\mb G}$. Suppose $\mb{X}=\mb{H}\backslash\mb{G}$ is a spherical $\mb{G}$-variety. Let $\nu:\mb{H}\to\U(1)$ be a character satisfying the following conditions.
    \begin{enumerate}
        \item The character $\nu$ is combinatorially trivial (Definition~\ref{defn:CombTrivial}, also cf.~Example~\ref{ex:CombTriv}).
        \item The induced Shimura datum on $\U(1)$ is non-trivial.
    \end{enumerate}
    Let $\sh{H}_\ell$ be the Hecke polynomial attached to $\mb{X}$ in the sense of Definition~\ref{def:Hecke}. Then there exists a collection of classes
    \[
        \big\{z_\m\in\h^{d+1}_\cont(\Sh_{\bG/E[\m]},\Z_p(d))\,\big|\,\m\in\mathscr{R}\big\}
    \]
    such that whenever $\m,\m\ell\in\mathscr{R}$, we have the tame norm relation
    \[
        \Tr_{E[\m]}^{E[\m\ell]}z_{\m\ell}=\sh{H}_{\ell}(\mathrm{Frob}_\lambda^{-1})\cdot z_\m.
    \]
    Moreover $z_1$ is the image of the special cycle under the continuous \'etale cycle class map.
\end{theorem}

The classes $z_\m$ come from Hecke translates of special cycles, as usual. The novelty here is that the existence of such Hecke translates is abstractly deduced by exploiting local harmonic analysis on spherical varieties.

\begin{remark}
    We informally explain some of the definitions occurring in the above theorem.
    \begin{itemize}
        \item  In the relative Langlands program, the period integral attached to $\mb{X}$ is conjecturally related to an $L$-function for automorphic representations of $\mb{G}$. The Hecke operator $\sh{H}_\ell$ gives the local factors of this $L$-function under the classical Satake isomorphism. In all cases considered in this paper, this is the ``standard'' $L$-function in that setting.

        It is important to note that this is \emph{different} from the ``$L$-function of $X$'' defined in \cite[Definition 7.2.3]{SakSpherical}. In its notation, we only keep (a part of) the denominator and disregard the adjoint $L$-functions.
        \item  The conditions on $\nu$ allow us to interpret cohomology classes for $\Sh_\bG$ over $E[\m]$ as classes for $\Sh_{\bG\times\U(1)}$ over the base field $E$. The combinatorially trivial condition guarantees that the $L$-values of all such character twists are described using the same period integral.
    
        In the Friedberg--Jacquet case ($\U(n)\times\U(n)\hookrightarrow\U(2n)$), these conditions explain the appearance of the character $\frac{\det h_1}{\det h_2}$. However, in the triple product case, the Ichino formula only describes the central value, and such a character does not exist. This is a well-known obstacle for extending the Gross--Kudla--Schoen cycle to an Euler system.
    \end{itemize}
\end{remark}

We can in fact replace $\Z_p$ by more general coefficient systems. We also have a version of the above theorem for the pushforward of Eisenstein classes (Example~\ref{ex:SiegelUnits}), which recovers the final four entries in the table. However, they depend on numerical consequences of the results of Sakellaridis--Wang \cite{SakWang}, which are not yet available in the mixed characteristics setting. 

\subsection{Example: Twisted Friedberg--Jacquet setting}\label{ss:TwistedFJ}
We now describe the third row of Table~\ref{tab:Examples} in details and highlight some new features of our approach. In this setting, our Euler system is new.

We use the set-up in \cite{LXZ-TwistedFL}. Let $F_0$ be a totally real field. Let $E_0/F_0$ be a totally real quadratic extension, and let $F/F_0$ be a CM extension. Let $E=E_0\otimes_{F_0} F$, so it is a biquadratic extension of $F_0$. Let $F'/F_0$ be the remaining quadratic extension of $F_0$ contained in $E$. Let $\sigma$ denote the non-trivial element of $\Gal(F/F_0)$. In summary, we have the field diagram
\begin{center}
    \begin{tikzpicture}[scale=0.8]
    \node (F0) at (0,0) {$F_0$};
    \node (E0) at (1.5,1.5) {$E_0$};
    \node (F') at (0,1.5) {$F'$};
    \node (F) at (-1.5,1.5) {$F$};
    \node (E) at (0,3) {$E$};

    \draw (F0) -- (E0) node [pos=0.5, below right] {Real, $\tau$} -- (E);
    \draw (F0) -- (F) node [pos=0.5, below left] {CM, $\sigma$} -- (E);
    \draw (F0) -- (F') -- (E);
    \end{tikzpicture}
\end{center}

Let $(B,\ast)$ be a division algebra over $F$ of dimension $(2n)^2$ with an involution $\ast$ of second kind. This forces $B$ to be a matrix algebra at the archimedean places. Suppose that $(B,\ast)$ has signature $(1,2n-1)$ at one archimedean place and $(0,2n)$ at the other archimedean places. Suppose further that there is an embedding $E\hookrightarrow B$ whose image is fixed by $*$, then the centralizer $B^E$ is a division algebra over $E$ of dimension $n^2$.

Define the following algebraic groups over $F_0$
\[
    \bH=\Res_{E_0/F_0}\U(B^E)=\{g\in (B^E)^\times| g^*g=1\},\quad \bG=\U(B)=\{g\in B^\times|g^*g=1\}
\]
They both have Shimura varieties with reflex fields (contained in) $F$ and dimensions
\[
    \dim_F \Sh_{\mb{H}}=n-1,\quad \dim_F \Sh_{\mb{G}}=2n-1.
\]
Therefore, the basic dimension numerology is satisfied.

Let $\pi$ be a symplectic automorphic representation of $\GL_{2n}(\A_{F_0})$ whose base change to $F$ is cuspidal. Let $\rho_\pi$ be the $2n$ dimensional Galois representation attached to $\pi$. Our main result in this case constructs the tame part of a JNS Euler system for the CM extension $F'/F_0$.
\begin{theorem}\label{thm:TwistedFJ}
    There is a JNS Euler system for the decomposable representation
    \[
        (\rho_\pi\oplus\rho_\pi\eta_{F/F_0})|_{\Gal_{F'}}
    \]
    whose base class is the $p$-adic \'{e}tale realization of the special cycle corresponding to the embedding $\Sh_{\mb{H}}\hookrightarrow\Sh_{\mb{G}}$.
\end{theorem}

This theorem is a standard consequence of Theorem~\ref{thm:1}, and we now explain the conditions of that theorem in this setting. Since our Euler system is for $F'/F_0$, the augmentation torus is $\mb{T}=\U(1)_{F'/F_0}$. The reduced norm defines a character $\mathrm{Nrd}:\U(B^E)\to\U(1)_{E/E_0}$ over $E_0$. Using it, we obtain a character
\[
    \nu:\mb{H}\to(\Res_{E_0/F_0}\U(1)_{E/E_0})/\U(1)_{F/F_0}\simeq\mb{T},
\]
where the final map is the norm map from $E$ to $F'$. Using it, we define the augmented pair of groups
\[
    \mb{H}\hookrightarrow\widetilde{\mb{G}}=\bG\times\mb{T},\quad h\mapsto(h,\nu(h))
\]
It is easy to verify using Example~\ref{ex:CombTriv} that the character $\nu$ is combinatorially trivial. 

Let $\ell$ be a place of $F_0$. Away from finitely many bad places, there are four cases depending on its splitting behaviours in $F$ and $E_0$. The pairs $(\mb{G},\mb{H})$ have the following descriptions after localizing at $\ell$:
\begin{enumerate}[label=(\roman*)]
    \item Splits in $F$ and $E_0$: $(\GL_{2n,F_0}, \GL_{n,F_0}\times\GL_{n,F_0})$,
    \item Inert in $F$ and $E_0$: $(\U_{2n,F/F_0}, \Res_{F/F_0}\GL_{n,F})$.
    \item Splits $F$, inert in $E_0$: $(\GL_{2n,F_0}, \GL_{n,E_0})$,
    \item Inert in $F$, split in $E_0$: $(\U_{2n,F/F_0}, \GL_{n,F_0}\times\GL_{n,F_0})$.
\end{enumerate}

For a JNS Euler system, we need the place $\ell$ to be split in $F'$, which corresponds to cases (i) and (ii) above. For arithmetic applications without conditions on $p$, it is necessary to consider 100\% of such primes. Case (i) was handled in \cite{GrahamShah} using the zeta integral method. However, in case (ii), the local multiplicity is $2^n$. So the zeta integral method cannot be applied directly. Moreover, we are not aware of any theory of local zeta integrals for this case.

Instead, we reduce the problem to a question of unramified local harmonic analysis on the symmetric space $\mb{H}\backslash\mb{G}$, which has been studied extensively. In case (i), we use the general result of Sakellaridis \cite{SakSpherical}. In case (ii), we compute the spectral decomposition ourselves following the Casselman--Shalika--Hironaka method developed in \cite{Hir10,Off04} (Theorem~\ref{thm:Satake non-split FJ}). In both cases, the denominator of the Plancherel formula contains the centre of the standard $L$-function for $\mathrm{BC}_{F_0}^F(\pi)$. This is the $L$-function attached to $\mb{X}$ in Theorem~\ref{thm:1}.

\begin{remark}
    Even though we still need to do local calculations, our approach has the following advantages compared to the earlier zeta integral approach, even in cases where it is available:
    \begin{enumerate}
        \item The harmonic analysis of $C_c^\infty(X,\C)$, where $X$ is a symmetric variety, is of independent interest.
        \item In previous works, choosing the test vector and computing the zeta integral both require substantial trial and error, whereas the spectral decomposition of $C_c^\infty(X,\C)^K$ in many cases can be systematically computed using the Casselman--Shalika method.
    \end{enumerate}
\end{remark}

\subsection{Idea of proof}
Our main result is a combination of the following two steps which are completely different in nature:
\begin{enumerate}
    \item Construction of a ``motivic theta series'' (Definition~\ref{defn:MotivicTheta}). 
    \item Local harmonic analysis on spherical varieties (Proposition~\ref{prop:MainCorollary}).
\end{enumerate}
We now explain each item in turn.
\subsubsection{Motivic theta series}
By definition, this is a $\mb{G}(\A^{p\infty})$-equivariant map between an ad\`{e}lic function space and certain ``motivic classes''. In the settings considered in this paper, the continuous \'etale cohomology group plays the role of this space of motivic classes. For arithmetic applications, it is important to have integral coefficients on the target cohomology group.

For the pushforward construction, versions of this map with \emph{rational} coefficients have appeared, for example in \cite[Definition 9.2.3]{LSZ-U3} and \cite[Proposition 9.14]{GrahamShah}. Our first main idea is that by considering the correct function space, there is a natural integral refinement of this construction.
\begin{observation}[Propositions~\ref{prop:cycle}, \ref{prop:Eisenstein}]
    In many cases, including the ones cited above, a more natural statement is that there is an \emph{integral}, $\mb{G}(\A^{p\infty})$-equivariant map
\[
    C_c^\infty(\mb{X}(\A^{p\infty}),\Z_p)\to\{\text{Integral motivic classes}\},
\]
where $\mb{X}$ is a spherical variety.
\end{observation}
\noindent More precisely, in the cycles case, we take $\mb{X}=\mb{H}\backslash\mb{G}$, and there is the commutative diagram
\[
\begin{tikzcd}
    C_c^\infty(\bG,\Q_p)_{\mb{H}}\rar & \h_\cont^{2d}(\Sh_\bG,\Q_p(d))\\
    C_c^\infty(\mb{X},\Z_p)\uar[hook]\rar & \h_\cont^{2d}(\Sh_\bG,\Z_p(d))\uar[hook]
\end{tikzcd}
\]
The previous constructions worked with the top arrow directly, but the pre-image of the natural lattice in cohomology did not have an easy description. Our observation is that the bottom arrow exists (Proposition~\ref{prop:cycle}), making $C_c^\infty(\mb{X},\Z_p)$ the natural integral lattice in the coinvariant space. The cases involving Eisenstein classes are treated similarly (Proposition~\ref{prop:Eisenstein}). In these cases, we take $\mb{X}$ to be the vector bundle
\[
    \mb{X}=\mb{G}\times^{\GL_2}\mathtt{std}=\GL_2\backslash(\mb{G}\times\mathtt{std}),
\]
instead of the previously used coinvariant space $C_c^\infty(\mb{G}\times\mathtt{std},\Z_p)_{\GL_2}$. Here, $\mathtt{std}$ is the standard two-dimensional representation of $\GL_2$.\\

This observation was already used by the third named author in \cite{LaiSkinner} to bypass some Iwasawa-theoretic arguments needed previously to deal with torsions. This was also independently observed by A.~Groutides, who used it to study the optimality of Hecke operators appearing in tame norm relations \cite{GroutidesRS}.

\subsubsection{Local harmonic analysis}
Once such a motivic theta series exists, the problem of tame norm relations is entirely reduced to constructing test vectors satisfying norm relations in $C_c^\infty(X,\Z_p)$, where $X=\mb{X}(F_\ell)$.

The first question to understand is the origin of the field extension $E[\m]$. Inspired by \cite{LoefflerSpherical,GrahamShah}, we introduce an augmented group $\widetilde{\bG}=\mb{G}\times\mb{T}$, where $\mb{T}$ is a one-dimensional torus that is supposed to parameterize character twists. By considering the desiderata of such an augmentation, we are led to the definition of a combinatorially trivial $\mb{T}$-bundle (Definition~\ref{defn:CombTrivial}). This ensures that $\pi$ is $\mb{X}$-distinguished implies that $\pi\times\chi$ is $\widetilde{\mb{X}}$-distinguished for any $\chi\in\hat{\mb{T}}$, giving rise to a family of character twists.

Using this bundle, we can define two level structures: for $i=0,1$, let $K^i=\bG(\Or)\times\mb{T}^i$, where $\mb{T}^0=\mb{T}(\Or)$ and $\mb{T}^1$ is the subset which is congruent to 1 modulo the uniformizer. The existence of tame norm relations is reduced to the following question: given a Hecke operator $\sh{H}_\ell$ and a ``basic element'' $\Phi_0$,
\begin{equation}\label{eqn:???}
    \text{is it true that }\sh{H}_\ell\cdot\Phi_0\in \Tr_{K^0}^{K^1}C_c^\infty(X,\Z_p)^{K^1}?\tag{\textdagger}
\end{equation}
This is \emph{stronger} than just requiring the function $\sh{H}_\ell\cdot\Phi_0$ to take values in $\Z_p$. Indeed, in the extreme case where the $K^0$-orbit of a point $x\in X$ coincides with its $K^1$-orbit, we need the value at $x$ to be divisible by the index $[K^0:K^1]=\ell-1$. We give a necessary condition for when this kind of additional divisibility requirements can occur, in terms of the geometry of $\mb{X}$ (Proposition~\ref{prop:Volume}). This is done by extending the proof of the generalized Cartan decomposition due to Gaitsgory--Nadler \cite[Theorem 8.2.9]{GN-Dual} and Sakellaridis \cite[Theorem 2.3.8]{SakANT}, 

To verify these divisibility conditions, we need to compute $\sh{H}_\ell\cdot\Phi_0$. Since $\sh{H}_\ell$ is described using its Satake image, we use works on the \emph{relative} Satake isomorphism by Sakellaridis \cite{SakSpherical} and Sakellaridis--Wang \cite{SakWang}. Our second main idea is the following.
\begin{observation}[cf.~Proposition~\ref{prop:DivisibleT}]
    The structure of the inverse relative Satake transform implies additional divisibility properties.
\end{observation}
\noindent We use the Heegner points case as an instructive example. At a split place, the spherical variety is $\mb{X}=\G_m\backslash\mathrm{PGL}_2$. The generalized Cartan decomposition identifies $\mb{X}(F)/\mb{G}(\Or)$ with the non-positive integers. In Theorem~\ref{thm:InverseSatake}, the $L$-factor is
\[
    \hat{\sh{L}}_{(X)}=(1-\ell^{-\frac{1}{2}}x)(1-\ell^{-\frac{1}{2}}x^{-1}),
\]
corresponding to the centre of the standard degree 2 $L$-function. The right hand side of Theorem~\ref{thm:InverseSatake} is
\[
    (1-\ell^{-\frac{1}{2}}x)(1-\ell^{-\frac{1}{2}}x^{-1})\cdot\frac{1-x^{-2}}{(1-\ell^{-1/2}x^{-1})^2}=(1-x^{-2})\cdot\frac{1-\ell^{-1/2}x}{1-\ell^{-1/2}x^{-1}}.
\]
This can be expanded as a formal Laurent series in $x$. The conclusion of Theorem~\ref{thm:InverseSatake} is that the function $\phi\in C_c^\infty(\mb{X}(F)/\mb{G}(\Or),\Z_p)=C_c(\Z_{\leq 0},\Z_p)$ whose relative Satake transform is $\hat{\sh{L}}_{(X)}$ is exactly given by the non-positive coefficients of this formal Laurent series, up to an explicit half power of $\ell$. By keeping track of these half powers carefully, it is easy to show that the values of $\phi$ are all polynomials in $\ell$.

From the geometric volume consideration above, we need to show that $\phi(0)\in(\ell-1)\Z_p$. To do this, we view $\ell$ as a formal variable and specialize at $\ell=1$. The rational function becomes the polynomial
\[
    (1-x^{-2})\cdot\frac{1-x}{1-x^{-1}}=x^{-1}-x
\]
The coefficient of $x^0$ is clearly 0, which proves the desired divisibility by $\ell-1$. In general, instead of expanding out this polynomial directly, we observe that it is anti-symmetric under the reflection $x\leftrightarrow x^{-1}$. This is the strategy used to prove the general Proposition~\ref{prop:DivisibleT}.

Combining these two computations, we note that this automatic divisibility is stronger than the geometric requirements, so the answer to \eqref{eqn:???} is yes, and we obtain the tame norm relation!

\subsection{Arithmetic applications}
Starting from Theorem~\ref{thm:1}, there are well-known methods in the literature to extract arithmetic consequences from it. We state one such example application, but it is not the focus of our paper.

The next corollary constructs the tame part of a JNS Euler system in Galois cohomology, following a standard Abel--Jacobi map procedure described in \S\ref{ss:MotivicES}.

\begin{cor}\label{cor:ActualES}
    Suppose in addition that
    \begin{itemize}
        \item $\mb{G}$ is anisotropic modulo centre.
        \item Condition (C') of \cite{MorelSuh} holds for $\mb{G}$.
        \item Kottwitz's conjecture \cite[Conjecture 5.2]{BlasiusRogawski} holds for the middle degree cohomology of $\Sh_{\bG/\bar{E}}$.
    \end{itemize}
    Let $\pi$ be a stable cohomological automorphic representation of $\mb{G}(\A_F)$ distinguished by $\mb{X}$. Let $\rho_\pi$ be the $p$-adic Galois representation attached to $\pi$ and the Shimura cocharacter for $\Sh_\bG$.
    
    Under the above set-up, there is a lattice $T_\pi$ in $\rho_\pi$ and a collection of Galois cohomology classes
    \[
        \{c_\m\in\h^1(E[\m],T_\pi)\,|\,\m\in\mathscr{R}\}
    \]
    forming the tame part of a JNS Euler system. In other words, whenever $\m,\m\ell\in\mathscr{R}$, we have the tame norm relation
    \[
        \Tr_{E[\m]}^{E[\m\ell]}c_{\m\ell}=P_\lambda(\Fr_\lambda^{-1})c_\m,
    \]
    where $P_\lambda(X)=\det(1-X\Fr_\lambda|\rho_\pi)$ is the characteristic polynomial of $\Fr_\lambda$.
\end{cor}

\begin{remark} We briefly explain the roles of the conditions, which are not too serious thanks to a large body of work in the area.
    \begin{itemize}
        \item The anisotropic modulo centre is a simplifying condition so that the Shimura variety is compact and we can directly apply the above cited works.
        \item Condition (C') is a collection of statements related to Arthur's conjecture for $\mb{G}$, If $\bG$ is a unitary or orthogonal group, then condition (C') is known (cf.~the discussion after Remark 1.6 in \cite{MorelSuh}). Their result is only used to modify our classes $z_\m$ to be null-homologous. In the unitary case, \cite[Proposition 6.9]{LL2021} also suffices for this purpose.
        \item Kottwitz's conjecture is used to show that $\rho_\pi$ actually contributes to the cohomology of $\Sh_{\bG/\bar{E}}$. In all of our cases, the Shimura variety is of abelian type, and what we need is follows from the work of Kisin--Shin--Zhu \cite{KSZ}.
    \end{itemize}
\end{remark}

By combining our construction with the results of \cite{JNS}, we obtain the following implication.
\begin{theorem}
    With notations as in the previous corollary, suppose the Galois representation $\rho_\pi:\Gal_E\to\GL(V_\pi)$ satisfies the following conditions.
    \begin{enumerate}
        \item $\rho_\pi$ is absolutely irreducible.
        \item There exists $\sigma\in\Gal_{E[1](\mu_{p^\infty})}$ such that $\dim V_\pi/(\sigma-1)V_\pi=1$.
        \item There exists $\gamma\in\Gal_{E[1](\mu_{p^\infty})}$ such that $V_\pi/(\gamma-1)V_\pi=0$.
    \end{enumerate}
    Then
    \[
        c_1\neq 0\implies\dim\h^1_f(E,\rho_\pi)=1
    \]
\end{theorem}

\begin{remark}
    \begin{enumerate}
        \item Since we are not using wild norm relations, there is no specific hypothesis on the prime $p$ in any of the above results.
        \item In the twisted Friedberg--Jacquet case, the JNS Euler system argument does not apply directly to such decomposable representations, but we hope a suitable modification of it could yield applications towards rank 1 cases of the Bloch--Kato conjecture for $\rho_\pi$.
        \item In the Eisenstein class cases, our construction gives classes over cyclotomic extensions (instead of anticyclotomic ones). Our resulting Euler system is the type considered by Rubin, cf.~\cite[Definition 10.6.1]{LSZ}. Since we are not considering wild norm relations, the Iwasawa cohomology in \emph{loc.~cit.} should be replaced by the usual cohomology.
    \end{enumerate}
\end{remark}

\subsection{Related works}
\subsubsection{Zeta integral approach}
In all works cited in Table~\ref{tab:Examples} (except \cite{LLZ}), the tame norm relation is verified using the zeta integral method first developed by Loeffler--Skinner--Zerbes \cite{LSZ}. In this method, one writes down a carefully chosen candidate for the test vector in $C_c^\infty(X,\Z_p)^{K^1}$ and verifies that its trace is equal to $\sh{H}_\ell\cdot\Phi_0$ by an explicit, often intricate, zeta integral computation. It also relies crucially on local multiplicity one. Our method avoids both the computations and the local multiplicity one hypothesis.

\subsubsection{Other approaches}
The idea of using spherical varieties in the construction of Euler systems goes back to Cornut \cite{CornutES}. Unfortunately, there is a gap in his construction related to the definition of Hecke operators, as explained in \cite{ShahHeckeDescent}. From the point of view of our framework, his setting presents additional features worthy of further investigation (cf.~Remark~\ref{rmk:1}).

In certain spherical settings where classes are obtained by a pushforward construction, Loeffler gave a systematic construction of wild norm relations \cite{LoefflerSpherical}. He also considered more general cases where a mirabolic subgroup has an open orbit (Definition 4.1.1 of \emph{op.~cit.}). In many such cases, we reinterpret the mirabolic subgroup as the point stabilizer of a naturally occurring affine \emph{inhomogeneous} spherical variety. However, in some examples such as $(\mathrm{GSp}_4,\GL_2\times_{\G_m}\GL_2)$, only a parabolic subgroup of $\bG$ has an open orbit. Our method does not handle this case. This is an Eisenstein degeneration of the spherical pair $(\mathrm{SO}_4\times\mathrm{SO}_5,\mathrm{SO}_4)$, and it would be interesting to understand how this can be interpreted as an operation on spherical varieties.

During the preparation of this work, the preprint \cite{ShahZeta} was posted. Shah considered a similar local question as \eqref{eqn:???}, phrased using double cosets on the group $\mb{G}$. At this point, the \emph{classical} Satake isomorphism is used in \emph{op.~cit.}, leading to complicated expressions involving Kazhdan--Lusztig polynomials, which need to be computed on a case-by-case basis. However, his method can treat certain non-spherical cases, which at the present falls outside the conjectural framework of \cite{BZSV}.

\subsubsection{Removing conditions}\label{1.4.1}
To make our results unconditional, we would need to have the function-level results of \cite{SakWang} for mixed characteristic local fields. Such a statement should follow from motivic integration techniques, along the lines of \cite{CHL-FL}. The third named author plans to write a short note on this matter in the future, though it would certainly be more desirable to have a sheaf-level statement and proof.

It would also be useful to more systematically develop our results in the non-split setting. The necessary local harmonic analysis results should be within reach of current methods. For example, \cite{SakWang} only assumes that the group is quasi-split, and \cite{MurilloII} handles many symmetric cases.

\subsubsection{Vista}
Perhaps the deepest question raised by our work is to understand a space of the form
\[
    \Hom_{\bG(\A)}(\mathrm{Fun}(\mb{X}(\A),\Z),\{\text{Integral motivic classes}\})
\]
in larger generality. Indeed, replacing ``integral motivic classes'' with ``automorphic functions'', then this space (though not its two constituent pieces) has a conjectural dual description in the relative Langlands program. In addition to the pushforward of cycles or motivic classes described above, the arithmetic theta lifting (cf.~\cite{LiuATL-I, LL2021, DisegniES}) should also be part of this framework. It would also be interesting to understand its relation with the recent relative cohomology construction of Sangiovanni--Skinner, cf.~\cite{SkinnerMarco}.

\subsection*{Acknowledgments}
It is clear that most of this paper is built on the fundamental works of Yiannis Sakellaridis. We would like to thank him for insightful conversations and for pointing out a major gap in an earlier version of the paper.

We would also like to thank David Ben-Zvi, Ashay Burungale, David Loeffler, Christopher Skinner, Ye Tian and Wei Zhang for their interests and helpful discussions.

The first and second named authors are  supported by the National Key R$\&$D Program of China No. 2023YFA1009702 and the National Natural Science Foundation of China No. 12371012. The second named author is also supported by  National Natural Science Foundation of Beijing, China No. 24A10020.

\section{Spherical varieties}\label{sec:Spherical}
Let $F$ be a field of characteristic 0, not necessarily algebraically closed. Unless otherwise specified, everything in this section will be defined over $F$.

Let $\bG$ be a split reductive group with a Borel subgroup $\mb{B}$. Let $\mb{X}$ be a variety with a right $\bG$-action. Recall that $\mb{X}$ is \emph{spherical} if the action of $\mb{B}\subseteq\bG$ on $\mb{X}$ has an open orbit. We make the following assumptions
\begin{enumerate}
  \item $\mb{X}$ is smooth, affine, connected.
  \item $\mb{X}$ has no root of type N (cf.\ \S\ref{ss:Classification}).
  \item Every $\mb{B}$-orbit on $\mb{X}_{/\overline{F}}$ contains an $F$-point.
  \item $\mb{X}$ has an invariant $\mb{G}$-measure.
\end{enumerate}
The assumptions imply that there is a unique open $\mb{B}$-orbit even on the level of $F$-points. We denote this orbit by $\mathring{\mb{X}}$. Fix once and for all a point $x_0\in\open{\mb{X}}(F)$. Let $\bH$ be the stabilizer of $x_0$. Let $\mb{X}^\bullet$ be the open $\mb{G}$-orbit in $\mb{X}$, so $\mb{X}^\bullet\simeq\mb{H}\backslash\bG$, and it contains $\mathring{\mb{X}}$ as an open dense subset.

\begin{remark}
  The condition that $\bG$ is split is present in most of the literature on spherical varieties. Even though our groups are non-split globally, we will only apply the results of this section when $F$ is a local field where the group $\bG$ splits

  Assumptions (1) and (2) roughly correspond to the assumptions imposed in \cite{BZSV} in the case of polarized Hamiltonian varieties. There have been progresses towards the unramified Plancherel formula without some of these hypothesis, for example \cite{SakWang} for certain singular varieties, and \cite{MurilloII} for certain varieties with roots of type N. It would be interesting to see if the behaviour observed in this paper still holds in these settings.
  
  Assumption (3) is purely a matter of convenience for the present paper to rule out spherical roots of type T non-split, which requires a separate analysis. In a future work, we plan to remove it and moreover consider the case when $\bG$ is not split.

  Assumption (4) is also included only for simplicity of notation. It holds if $\mb{H}$ is reductive or if $\mb{X}$ is of the form $\mb{G}\times^{\mb{H}'}V$, where $\mb{H}'$ is reductive and $V$ is a linear representation of $\mb{H}'$. These are the two cases needed for the Euler system constructions in the paper. In general, one may assume there is an $\mb{G}$-eigenmeasure after trivially modifying \cite[\S 3.8]{SakComp}, and our formulae in \S\ref{ss:Satake} need to be modified by the corresponding character in a well-understood way.
\end{remark}

\subsection{Structure theory}\label{sec:Structure}
We now recall some general results from the theory of spherical varieties, stating only what is needed for the present work. We refer to \cite{SakSpherical,SVSpherical} and references found therein for a more systematic development with proofs.

\subsubsection{Notations}
Let $\mb{P}(\mb{X})$ be the subgroup of $\mb{G}$ fixing the open orbit. It contains the Borel subgroup, so it is a parabolic subgroup. Choose a good Levi subgroup $\mb{L}(\mb{X})$ as in \cite[\S 2.1]{SakSpherical}, and let $\mb{A}$ be a maximal torus of $\bG$ contained in $\mb{B}\cap\mb{L}(\mb{X})$. Define the torus
\[
  \mb{A}_X=\mb{A}/(\mb{A}\cap\bH)
\]
Write $\Lambda_X$ for its cocharacter lattice, and let $\mathfrak{a}_X=\Lambda_X\otimes_\Z\Q$. Similarly define $\mathfrak{a}$ as the coroot lattice of $\mb{G}$ tensored with $\Q$, then we have a quotient map $\mathfrak{a}\twoheadrightarrow\mathfrak{a}_X$. There is a natural cone $\sh{V}\subseteq\mathfrak{a}_X$ containing the image of the \emph{negative} Weyl chamber. Let $\Lambda_X^+=\sh{V}\cap\Lambda_X$. Its elements will be called $\mb{X}$-anti-dominant.

Let $\sh{V}^\perp$ be the negative-dual cone to $\sh{V}$ in $X^*(\mb{A}_X)\otimes_\Z\Q$, then $\sh{V}^\perp$ is strictly convex. 
In \cite[\S 6.1]{SakSpherical}, Sakellaridis defined a based root system $\Phi_X$ whose set of simple roots $\Delta_X$ lie on extremal rays of $\sh{V}^\perp$ intersected with $X^*(\mb{A}_X)$. Elements of $\Delta_X$ are called (normalized) \emph{spherical roots} of $\mb{X}$. There is a canonical embedding $X^*(\mb{A}_X)\hookrightarrow X^*(\mb{A})$, allowing us to view spherical roots as characters of $\mb{A}$. The Weyl group of $\Phi_X$ is the \emph{little Weyl group}, denoted by $W_X$. It is canonically contained in the Weyl group of $\mb{G}$, which we will denote by $W$.

Knop and Schalke defined a dual group $\check{G}_X$ whose coroot system is $(X^*(\mb{A}_X),\Phi_X)$ \cite{KnopSchalke}. It is a subgroup of the Langlands dual group $\check{G}$ of $\mb{G}$. We have an equality $\check{G}_X=\check{G}$ only if for all simple roots $\alpha$, the pair $(\mathring{\mb{X}},\alpha)$ is of type T in the classification below. Spherical varieties $\mb{X}$ with this property are said to be \emph{strongly tempered}.

\subsubsection{Classification of roots}\label{ss:Classification}
Let $\alpha$ be a simple root of $\mb{G}$. Let $\mb{P}_\alpha$ be its associated standard parabolic subgroup, with radical $\mathcal{R}(\mb{P}_\alpha)$. Let $\mb{Y}$ be a $\mb{B}$-orbit in $\mb{X}$. The geometric quotient
\[
    \mb{Y}\mb{P}_\alpha/\mathcal{R}(\mb{P}_\alpha)
\]
is a homogeneous spherical variety of $\mathrm{PGL}_2$. There are four cases.
\begin{itemize}
    \item Type G: $\ast=\mathrm{PGL}_2\backslash\mathrm{PGL}_2$.
    \item Type U: $\mb{S}\backslash\mathrm{PGL}_2$, where $\mb{S}$ is the subgroup of a Borel subgroup which contains the unipotent radical.
    \item Type T: $\mb{T}\backslash\mathrm{PGL}_2$, where $\mb{T}$ is a maximal torus.
    \item Type N: $\mb{N}(\mb{T})\backslash\mathrm{PGL}_2$, where $\mb{N}(\mb{T})$ is the normalizer of the torus. This case is excluded by our assumption.
\end{itemize}
We say the pair $(\mb{Y},\alpha)$ is of type G, U, or T according to this classification.

Somewhat confusingly, there is a related but separate classification of spherical roots. Note that it is also standard terminology that the spherical roots only refer to the \emph{simple} roots of the spherical root system $\Phi_X$. Let $\gamma$ be a spherical root. Under our standing assumptions, it has one of the two types.
\begin{itemize}
  \item Type T: This happens exactly if $\gamma$ is a root of $\mb{G}$.
  \item Type G: In all other cases, $\gamma=\alpha+\beta$, where $\alpha,\beta$ are orthogonal roots of $\mb{G}$, and they are simple roots in some choice of basis.
\end{itemize}
More details, as well as a proof of the above dichotomy, can be found in \cite[\S 6.2]{SakSpherical}. The ample examples there should illustrate the classification. We simply note that by the description in \emph{loc.~cit.}, if there is a pair $(\mb{Y},\alpha)$ of maximal rank of type T, then there is a spherical root of type T.

\subsection{Equivariant bundles}
For applications to Euler systems, it is necessary to consider certain torus bundles over $\mb{X}$.
\begin{definition}\label{defn:CombTrivial}
  Let $\widetilde{\mb{X}}\to\mb{X}$ be an $\bG$-equivariant $\mb{T}$-bundle, where $\mb{T}$ is a torus. It is \emph{combinatorially trivial} if its restriction to $\open{\mb{X}}$ is trivial as a $\mb{B}$-equivariant bundle.
\end{definition}

In particular, $\widetilde{\mb{X}}$ is a spherical variety for the group $\widetilde{\bG}:=\mb{G}\times\mb{T}$, and the stabilizer of any lift of $x_0$ is isomorphic to $\mb{H}$ by projection to the first factor. It is immediate that
\[
  \mb{A}_{\widetilde{X}}=\mb{A}_{X}\times\mb{T}
\]
as quotients of the maximal torus $\mb{A}\times\mb{T}$. As a result, all of the combinatorial data introduced in the previous subsection are either unchanged or change in a trivial way. This explains our terminology. In particular, we note that the classification of spherical roots is unchanged.

\begin{example}\label{ex:CombTriv}
  If $\mb{X}=\mb{H}\backslash\mb{G}$ is homogeneous, then the datum of an equivariant $\mb{T}$-bundle over $\mb{X}$ is equivalent to a character $\nu:\mb{H}\to\mb{T}$ by the recipe
  \[
    \widetilde{\mb{X}}=\widetilde{\mb{H}}\backslash(\mb{G}\times\mb{T}),\quad \widetilde{\bH}=\{(h,\nu(h))\,|\,h\in\mb{H}\}
  \]
  It is combinatorially trivial if and only if $\nu|_{\mb{H}\cap\mb{A}}=1$, where recall that $\mb{A}$ is a maximal torus in the Borel subgroup $\mb{B}$ such that $\mb{H}\mb{B}$ is open in $\mb{G}$. This is the condition imposed by Loeffler when he systematically studied wild norm relations in \cite[\S 4.6]{LoefflerSpherical}.

  Observe that if $\mb{H}$ contains a maximal unipotent subgroup of $\mb{G}$ (so $\mb{X}$ is horospherical), then any combinatorially trivial bundle is trivial, since $\nu$ has to be trivial on the unipotent part.
\end{example}

\begin{remark}\label{rmk:CombTriv}
  Combinatorially trivial in particular implies that
  \[
    \check{G}_{\widetilde{X}}=\check{G}_{X}\times\check{\mb{T}}.
  \]
  On the representation theory side, this means if $\pi$ is $\mb{X}$-distinguished, then $\pi\times\chi$ is $\widetilde{\mb{X}}$-distinguished for any character $\chi$ of the torus. So we are looking a family of character twists for $\pi$, exactly what is needed for Euler system constructions. If $\mb{T}=\G_m$, then this is related to having an ``$s$-variable'' in the $L$-function.
  
  The existence of such a bundle is unfortunately quite restrictive, and it rules out very interesting cases, including the triple product case $\mb{G}=\mathrm{SO}_3\times\mathrm{SO}_4$, $\mb{H}=\mathrm{SO}_3$. Indeed, the Ichino formula only calculates the central value, and $\mb{H}$ does not admit non-trivial characters in this case.
\end{remark}

The following technical proposition plays a crucial role in understanding the divisibility properties needed to characterize the trace map considered in the next section. Its proof consists of expanding the $\mb{B}$-trivialization over the open orbit to the entire variety using simple reflections.

\begin{prop}\label{prop:Triviality}
    Suppose $\mb{X}$ is homogeneous and has only spherical roots of type G, then every combinatorially trivial $\mb{T}$-bundle over $\mb{X}$ is trivial.
\end{prop}
\begin{proof}
    Let $\mb{Y}$ be a $\mb{B}$-orbit in $\mb{X}$ of maximal rank. Let $\alpha$ be a simple root of $\bG$ with associated parabolic $\mb{P}_\alpha$. The pair $(\mb{Y},\alpha)$ cannot have type T, since otherwise there would be a spherical root of type T. It follows that the pair is of type G or U. In particular, $\mb{Y}\mb{P}_\alpha$ is either $\mb{Y}$ or the disjoint union of an open $\mb{B}$-orbit and a closed $\mb{B}$-orbit. In the latter case, both $\mb{B}$-orbits have the same rank. It follows that all $\mb{B}$-orbits are of maximal rank, so all pairs $(\mb{Y},\alpha)$ are not of type T.

    Let $\widetilde{\mb{X}}$ be any equivariant $\G_m$-bundle. Suppose it trivializes as a $\mb{B}$-equivariant bundle over $\mb{Y}$ and $\alpha$ is a root such that $\mb{Y}$ is the open orbit in $\mb{Y}\mb{P}_\alpha$. Let $\mb{H}$ be the stabilizer of a point in $\mb{Y}$, so the bundle structure becomes a map $\nu:\mb{H}\to\mb{T}$. Consider the short exact sequence of groups
    \[
        1\to\mb{H}\cap\mathcal{R}(\mb{P}_\alpha)\to\mb{H}\cap\mb{P}_\alpha\to\big((\mb{H}\cap\mb{P}_\alpha)\mathcal{R}(\mb{P}_\alpha)\big)/\mathcal{R}(\mb{P}_\alpha)\to 1
    \]
    Since $\mathcal{R}(\mb{P}_\alpha)\subseteq\mb{B}$, the restriction of $\nu$ to the first term is trivial, so it descends to a character $\bar{\nu}$ on the quotient. On the other hand, this quotient is the stabilizer group in the $\mathrm{PGL}_2$-spherical variety $\mb{Y}\mb{P}_\alpha/\mathcal{R}(\mb{P}_\alpha)$, which is necessarily of type U. But then $\bar{\nu}$ is trivial by Example~\ref{ex:CombTriv}. Therefore, $\nu|_{\mb{H}\cap\mb{P}_\alpha}=1$. In other words, $\widetilde{\mb{X}}$ trivializes as a $\mb{B}_\alpha$-bundle over $\mb{Y}\mb{P}_\alpha$.

    Suppose that $\widetilde{\mb{X}}$ is combinatorially trivial, then we can construct a $\mb{B}$-equivariant section over $\mathring{\mb{X}}$. By the above discussion, this section extends to $\mathring{\mb{X}}\mb{P}_\alpha$ for any simple root $\alpha$. We may continue this process starting from the closed orbits in $\mathring{\mb{X}}\mb{P}_\alpha$. Since $\mb{X}$ is homogeneous, all $\mb{B}$-orbits are reached this way. Moreover, all of the sections glue since the $\mb{B}$-equivariant sections on a $\mb{B}$-orbit are unique up to a constant multiple. Therefore, we have shown that $\widetilde{\mb{X}}$ is trivial as a $\mb{B}$-equivariant bundle.

    In particular, $\widetilde{\mb{X}}$ is trivial as a line bundle, so the only $\mb{G}$-equivariance structure on it comes from a character $\chi:\mb{G}\to\mb{T}$. Its restriction to $\mb{B}$ is trivial by the discussion above, so $\chi$ itself must be trivial.
\end{proof}

\section{Local computations}\label{sec:Local}
We now specialize to the case where $F$ is a local field, with ring of integers $\Or$ and residue field $\F$. Let $q=\#\F$. Let $\varpi$ be a uniformizer. For any of the varieties denoted by bold letters, we will use the normal font to denote its $F$-points, so for example $X=\mb{X}(F)$.

The group $G$ has a natural smooth left action on the function space $C^\infty(X,\C)$. In this section, we recall the spectral decomposition of its unramified part, following the works of Sakellaridis and his collaborators. The goal is to observe certain automatic divisibility properties and match them with a corresponding phenomenon in geometry.

\subsection{Assumptions}\label{ss:LocalAssumption}
For all results in this section, we need to impose the following ``good reduction'' hypotheses. In the global setting, they hold for all but finitely many places.
\begin{assumption}
  Both $\bG$ and $\mb{X}$ extend to smooth schemes over $\spec\Or$, which we denote by the same letter. Moreover, all statements of \cite[Proposition 2.3.5]{SakANT} hold. In particular,
  \begin{enumerate}
    \item $\bG$ is reductive and $\mb{X}$ is affine.
    \item The chosen base point $x_0$ belongs to $\mb{X}(\Or)$.
    \item A local structure theorem for $\mb{X}$ and its compactifications hold.
  \end{enumerate}
\end{assumption}

We will write $K=\bG(\Or)$. By hypothesis, this is a hyperspecial maximal compact subgroup of $G$.

\subsection{Generalized Cartan decomposition}
Under the assumptions made above, we will state a generalized Cartan decomposition due to Gaitsgory--Nadler in the equal characteristics case \cite[Theorem 8.2.9]{GN-Dual} and adpoted to the mixed characteristics case by Sakellaridis \cite[Theorem 2.3.8]{SakANT}. Recall that $\Lambda_X\simeq\mb{A}_{\mb{X}}(F)/\mb{A}_{\mb{X}}(\Or)$ contains an $\mb{X}$-anti-dominant monoid $\Lambda_X^+$. 
\begin{prop}[Genralized Cartan decomposition]\label{prop:GeneralCartan}
    For each $\check{\lambda}\in\Lambda_X^+$, fix a representation $x_{\check{\lambda}}\in A_X$, viewed as an element of $X$ by the orbit map through $x_0$. Then there is a disjoint union decomposition
    \[
        X^\bullet=\bigsqcup_{\check{\lambda}\in\Lambda_X^+}x_{\check\lambda} K.
    \]
\end{prop}

\subsubsection{Representation of functions}
Since $\mb{X}-\mb{X}^\bullet$ is Zariski closed in $\mb{X}$, a function $\phi\in C_c^\infty(X,\C)$ is uniquely determined by its restriction to $X^\bullet$, where it no longer has to be compactly supported. By the generalized Cartan decomposition, we have a canonical isomorphism $C^\infty(X^\bullet,\C)^K\iso\C[[\Lambda_X^+]]$ defined by
\begin{equation}\label{eqn:PowerSeries}
    \phi\mapsto\sum_{\check{\lambda}\in\Lambda_X^+}\phi(x_{\check{\lambda}})e^{\check{\lambda}}\in\C[[\Lambda_X^+]]
\end{equation}
where $e^{\check{\lambda}}$ is a formal symbol representing the monoid element $\check{\lambda}$. We will use such a formal power series to represent functions in $C_c^\infty(X,\C)$ in the future.

\subsubsection{Interaction with bundle}  
Let $\widetilde{\mb{X}}\to\mb{X}$ be a combinatorially trivial $\G_m$-bundle defined over $\spec\Or$. Let
\[
  J_0=\Or^\times,\quad J_1=\{x\in F^\times\,|\,x\equiv 1\pmod{\varpi}\}.
\]
The maximal compact subgroup $K^0:=\widetilde{\bG}(\Or)$ is equal to $\mb{G(\Or)}\times J_0$. For global reasons, we are also interested in the subgroup
\[
  K^1:=\bG(\Or)\times J_1\subseteq K^0.
\]
Clearly, $K^1$ is a normal subgroup of $K^0$ whose quotient is $\F^\times$, which has size $q-1$. The main result of this subsection is to describe the space $\widetilde{X}/K^1$.

\begin{definition}
  A coroot $\check\lambda$ is said to \emph{lie on a wall of type T} if there is a spherical root $\gamma$ of type T such that $\bra\check{\lambda},\gamma\ket=0$.
  
  A point $x\in X$ \emph{lies on a wall of type T} if it is in the orbit $x_{\check\lambda} K^0$, where $\check\lambda$ lies on a wall of type T.
\end{definition}

We are interested in understanding the image of the trace map
\[
    \Tr_{K_0}^{K^1}:C_c^\infty(\widetilde{X},\Z_p)^{K^1}\to C_c^\infty(\widetilde{X},\Z_p)^{K^0},
\]
which is equivalent to studying when restricting to $K^1$ splits an $K^0$-orbit in the generalized Cartan decomposition. We give a partial implication, which we suspect can be upgraded to an equivalence. The proof is based on considering asymptotic degenerations of the bundle $\widetilde{\mb{X}}\to\mb{X}$.

\begin{prop}\label{prop:Volume}
  Let $x\in\tX$. If the action of $K^1$ on $xK^0$ has fewer than $q-1$ orbits, then $x$ lies on a wall of type T.
\end{prop}
\begin{proof}
    In this proof, we will only work over the generic fibre, so the standard theory of spherical embeddings applies without comment. Further note that we are only working over the open $\mb{G}$-orbit $\mb{X}^\bullet$, so we will assume that $\mb{X}$ is homogeneous.
    
    Over the open orbit $\mathring{\mb{X}}$, fix a trivialization
    \[
        \widetilde{\mb{X}}\times_{\mb{X}}\mathring{\mb{X}}\simeq\mathring{\mb{X}}\times\G_m
    \]
    as a $\mb{B}$-equivariant bundle. Identify $A_X$ with a subset of $\mathring{X}$ using the orbit map. Let $x\in A_X^+$ be an element such that $K^1$ splits $xK^0$ into less than $q-1$ pieces. Then there exists $g\in K$ such that $(x,1)g=(x,\alpha)$, where $\alpha\not\equiv 1\pmod{\varpi}$. For any $n\geq 1$, we have
    \begin{equation}\label{eqn:orbit1}
        (x^n,1)g=(x^{n-1},1)(x,1)g=(x^n,\alpha),
    \end{equation}
    where we are using the identification of $A_X$ (as a group) with its orbit through the base point $x_0$ (a subset of $X$). It follows from this equation that the orbit $x^n K^0$ also does not split completely.

    Let $\check{\lambda}$ be the image of $x$ in $\Lambda_X^+$. Suppose for contradiction that $\check{\lambda}$ does not lie on a wall of type T, then the set $\Theta:=\{\gamma\in\Delta_X\,|\,\bra\check{\lambda},\gamma\ket=0\}$ consists only of spherical roots of type G. The idea of the proof is to show that the bundle trivializes at $\Theta$-infinity. This notion is introduced in \cite[\S 2.3]{SVSpherical}, and we recall some aspects of their construction.
    
    Let $\sh{F}_\Theta$ be the face in $\sh{V}$ consisting of vectors orthogonal to all spherical roots in $\Theta$, then it contains $\check{\lambda}$ in its relative interior. Choose a fan decomposition of $\sh{F}_\Theta$ so that $\check{\lambda}$ belongs to a face $\sh{F}$. Let $\overline{\mb{X}}$ be the spherical embedding defined by this fan decomposition, following Luna--Vust theory \cite[Theorem 3.3]{Knop91}. The face $\sh{F}$ corresponds to a $\mb{G}$-orbit $\mb{Z}$ satisfying the condition that $x^n$ converges to a point $x^\infty\in\mb{Z}$.
    
    The combinatorially trivial $\G_m$-bundle $\widetilde{\mb{X}}$ also extends to the boundary using the compactification given by the same fan decomposition. By \cite[Proposition 2.3.8(3)]{SVSpherical}, the spherical roots of $\mb{Z}$ are exactly $\Theta$, so $\mb{Z}$ is a spherical variety with only spherical roots of type G. Proposition~\ref{prop:Triviality} implies that the bundle over $\mb{Z}$ must split. On the other hand, taking limit as $n\to\infty$ in equation~\eqref{eqn:orbit1}, we see that $g$ translates between the $q-1$ disks above the point $x^\infty$. This is a contradiction, proving that $\check{\lambda}$ must lie on a wall of type T.
\end{proof}

\subsection{Relative Satake isomorphism}\label{ss:Satake}
In this subsection, we will introduce the relative Satake isomorphism. In the homogeneous case, the results are unconditional and due to Sakellaridis \cite{SakSpherical, SakInvSatake}. In the strongly tempered case, the analogous result holds in the equal-characteristic case \cite{SakWang}, and we conjecture that it carries over to the mixed characteristic case.

\subsubsection{Hecke algebra and Satake isomorphism}\label{ss:Normalization}
Fix the Haar measure on $G$ so that $K$ has volume 1, then we can define the Hecke algebra $\sh{H}(G):=C_c^\infty(K\backslash G/K,\C)$ whose algebra structure is given by convolution
\[
    (f_1\ast f_2)(x)=\int_G f_1(xg^{-1})f_2(g)\,dg
\]
This algebra acts on $C_c^\infty(X,\C)$ on the left in the natural way
\[
    (f\cdot\phi)(x)=\int_G \phi(xg)f(g)\,dg.
\]
Note that when $X=G$, this action is different from the convolution by an inverse.

Recall the classical Satake isomorphism, cf.~\cite{GrossSatake}
\begin{equation}\label{eqn:Satake}
    \sh{H}(G)\iso\sh{H}(A)^W\simeq\C[\Lambda]^W,\quad f\mapsto\hat{f}(t):=\delta(t)^{\frac{1}{2}}\int_N f(tn)dn
\end{equation}
where $\Lambda=X_*(\mb{A})=\mb{A}(F)/\mb{A}(\Or)$, and $\delta$ is the modulus character of the Borel subgroup containing $N$. By restriction, we have a natural ring homomorphism $\C[\Lambda]^W\to\C[\Lambda_X^+]$. Let $A^*$ be the set of unramified characters of $A$, so $A^*=X^*(\mb{A})\otimes_\Z\C^\times$ is an algebraic torus. There is a Fourier transform
\begin{equation}\label{eqn:Fourier}
    \C[\Lambda]\iso\C[A^*],\quad e^{\check{\lambda}}\mapsto (\chi\mapsto\bra\chi,\check{\lambda}\ket\in\C^\times)
\end{equation}
where as before, $e^{\check{\lambda}}$ is a formal symbol representing the monoid element $\check{\lambda}$. We will use the Fourier transform to identify the two spaces.

\subsubsection{Homogeneous case}\label{ss:HomogeneousCase}
In this subsection, suppose in addition that $\mb{X}$ is a homogeneous spherical variety. We will also assume Statements 6.3.1 and 7.1.5 from \cite{SakSpherical}. These are easy to verify for each given $\mb{X}$, and it is expected they always hold under the assumptions we have made already.

Let $A_X^*$ (resp.~$A^*$) be the set of unramified characters of $A_X$ (resp.~$A$). Since we are assuming all $\mb{B}$ orbits on $\mb{X}_{/\bar{F}}$ are defined over $F$, the tours $\mb{A}_X$ is split, and we have the canonical injection $A_X^*\hookrightarrow A^*$. We can alternatively describe $A_X^*$ as the $\C$-points of the dual torus $\check{\mb{A}}_X$.

Let $\delta_{(X)}$ (resp.~$\delta_{P(X)}$) be the modulus character for  $\mb{L}(\mb{X})\cap \mb{B}$ (resp.~parabolic subgroup $\mb{P}(X)$), cf.~\S\ref{sec:Structure}. The translate $\delta_{(X)}^{\frac{1}{2}}A_X^*$ defines a subvariety of $A^*$. Its ring of polynomial functions will be denoted by $\C[\delta_{(X)}^{\frac{1}{2}}A_X^*]$. We make the following canonical identifications
\begin{equation}\label{eqn:Shift}
    \C[\Lambda_X]\iso\C[A_X^*]\iso\C[\delta_{(X)}^{\frac{1}{2}}A_X^*],
\end{equation}
where the first map is the Fourier transform as defined in \eqref{eqn:Fourier}, and the second map is induced by translation.
\begin{theorem}[{\cite[Theorem 8.0.2]{SakSpherical}}]\label{thm:RelSatake}
  There exists an isomorphism
  \[
    \phi\mapsto\hat{\phi}:C_c^\infty(X,\C)^K\simeq\C[\delta_{(X)}^{\frac{1}{2}}A_X^*]^{W_X}
  \]
  such that
  \begin{enumerate}
    \item The image of the basic function
    \[
        \Phi_0:=\indf[\mb{X}(\Or)]
    \]
    is the constant function 1.
    \item Let $C_c^\infty(K\backslash G/K,\C)\simeq\C[A^*]^W$ be the classical Satake isomorphism, then the above isomorphism is equivariant with respect to the natural action on the left hand side and multiplication on the right hand side.
  \end{enumerate}
\end{theorem}

Let $\sh{L}\in C_c^\infty(K\backslash G/K,\C)$ be a Hecke operator with Satake transform $\hat{\sh{L}}\in\C[A^*]^W$. By restriction and using the identification \eqref{eqn:Shift}, this defines an element
\begin{equation}\label{eqn:ShiftedHecke}
    \hat{\sh{L}}_{(X)}\in\C[\delta_{(X)}^{\frac{1}{2}}A_X^*]\iso\C[\Lambda_X]
\end{equation}
We would like to compute the function $\sh{L}\cdot\Phi_0$ in terms of $\hat{\sh{L}}$, which can be done using the inverse Satake isomorphism \cite{SakInvSatake}. To state it, recall that Sakellaridis defined a multiset $\Theta_X^+$ in \cite[\S 7.1]{SakSpherical}. By our split hypothesis, all the signs $\sigma_{\check\theta}=1$ in the notation of \emph{loc.~cit.}, and the remaining two pieces of data can be combined as a representation of $\check{\mb{A}}\times\G_m$, cf.\ \cite[\S 9.3.5]{BZSV}. We will write an element of $\Theta_X^+$ in the form $(\check{\theta},d_{\check{\theta}})$, where $\check{\theta}$ is a cocharacter, and $d_{\check{\theta}}\in\Z$. 
\begin{theorem}\label{thm:InverseSatake}
    Let $\Theta_X^+$ be the multiset of weights of $\check{\mb{A}}\times\G_m$ described in \cite[\S 7.1]{SakSpherical}, then
    \[
        \delta_{P(X)}^{\frac{1}{2}}\cdot(\sh{L}\cdot\Phi_0):=\hat{\sh{L}}_{(X)}\cdot\frac{\prod_{\check{\gamma}\in\check{\Phi}_X^+}(1-e^{\check{\gamma}})}{\prod_{(\check{\theta},d_{\check\theta})\in\Theta_X^+}(1-q^{-\frac{1}{2}d_{\check{\theta}}}e^{\check{\theta}})}\bigg|_{\Lambda_X^+}
    \]
    under the identification~\eqref{eqn:PowerSeries}.
\end{theorem}

\begin{remark}\label{rmk:Cone}
  We make an important comment regarding the interpretation of the right hand side. It is a rational function, and we need to expand it as a formal Laurent series before taking the restriction to $\Lambda_X^+$. To make sense of this expansion, one needs to specify a cone $\sh{C}_X\subseteq\Lambda_X$. In the above theorem, this is the strictly convex cone (positive-)dual to the highest weights appearing in the $\bG$-module $F[\mathbf{X}]$, cf.~\cite[\S 6]{SakInvSatake}. In the homogeneous case, this is disjoint from $\Lambda_X^+$, so the right hand side describes a compactly supported function on $\Lambda_X^+$.
  
  The cone will not be present in our notations since our proofs do not depend on its choice, but it is nonetheless important to note that there is a fixed choice.
\end{remark}

\begin{proof}
    The result is \cite[Theorem 7.7]{SakInvSatake}, where it was proven under the additional assumption that $\mb{X}$ is wavefront. However, the only time the wavefront hypothesis is used in the derivation is to compute the normalization factor in the Plancherel measure. This exact normalization is not needed for the above result since the expressions do not depend on the choice of the $G$-invariant measure on $X$.
\end{proof}

\subsubsection{Strongly tempered case}
We no longer suppose $X$ is homogeneous. Instead, we work in the strongly tempered case, namely $\check{G}_X=\check{G}$. We state a conjectural analogue of Theorem~\ref{thm:InverseSatake}. In the case when $F$ is an equal-characteristic local field, the conjecture minus the final part is known by the main theorem of \cite{SakWang}.

In this case, we have $\mb{A}=\mb{A}_X$, so we can drop the subscript $X$ from all combinatorial data. Moreover, $\delta_{(X)}=1$ and $\delta_{P(X)}=\delta$. Finally, all elements of $\Theta_X^+$ have $d_{\check{\theta}}=1$ by the recipe constructing them. Therefore, the following is the exact analogue of Theorem~\ref{thm:InverseSatake}, with the final statement playing the role of \cite[Statement 7.1.5]{SakSpherical}.

\begin{conj}\label{conj:InverseSatake}
    Let $\Theta_X^+$ be the multiset of weights of $\check{\mb{A}}$ denoted by $\mathfrak{B}^+$ in \cite[Theorem 1.1.2]{SakWang}, then
    \[
        \delta^{\frac{1}{2}}\cdot(\sh{L}\cdot\Phi_0)=\hat{\sh{L}}\cdot\frac{\prod_{\check{\gamma}\in\check{\Phi}^+}(1-e^{\check{\gamma}})}{\prod_{\check{\theta}\in\Theta_X^+}(1-q^{-\frac{1}{2}}e^{\check{\theta}})}\bigg|_{\Lambda^+}
    \]
    Moreover, all weights of $\Theta_X^+$ are minuscule as coweights of $\mb{G}$.
\end{conj}
\begin{proof}[Proof in the equal-characteristic case]
    All references in this proof are to \cite{SakWang}. Define the Radon transform $\pi_!$ by the expression
    \[
        \pi_!\phi(t)=\int_{N}\phi(tn)\,dn.
    \]
    By unwinding definitions, it is easy to check that $\pi_!(\sh{L}\cdot\phi)=\hat{\sh{L}}\ast\pi_!\phi$, where the convolution takes place in the space of smooth (not necessarily compactly supported) functions on $\mb{A}$. In terms of the ring $\C[[\Lambda]]$, this convolution is just usual multiplication. Since $\pi_!\Phi_0$ is computed in Theorem 1.1.2, the conjecture follows by applying the theory of asymptotics, exactly as in the proof of Corollary 1.2.1.
\end{proof}

\subsubsection{Integral structure}
Let the assumptions be as in the previous subsection. We give an \emph{ad hoc} definition of an integral structure on $\C[\delta_{(X)}^{\frac{1}{2}}A_X^*]$, whose only purpose is to specify when a half power of $q$ can show up. This is abstractly explained by the analytic shearing construction in \cite[\S 6.8]{BZSV}. Let $\Z_{\{q\}}=\Z[q^{-1}]$.

\begin{definition}\label{defn:Shear}
  Let $\rho$ (resp.~$\rho_{(X)}$, $\rho_{P(X)}$) be the half sum of positive roots for $\mb{G}$ (resp.~$\mb{L}(\mb{X})$, $\mb{P}(\mb{X})$). Define $\Z_{\{q\}}[[A^*]]^{\shear}$ to be the functions in $\C[[A^*]]$ of the form
  \[
    \sum_{\check{\lambda}\in\Lambda_X}a_{\check{\lambda}}q^{\bra\check\lambda,\rho\ket}e^{\check{\lambda}}
  \]
  where $a_{\check{\lambda}}\in\Z_{\{q\}}$. We are again identifying $\C[[A^*]]$ with $\C[[\Lambda]]$ by \eqref{eqn:PowerSeries}.

  The ring $\Z_{\{q\}}[[\delta_{(X)}^{\frac{1}{2}}A_X^*]]^{\shear}$ consists of restrictions of functions in $\Z[[A^*]]^{\shear}$. 
\end{definition}

\begin{lemma}\label{lem:Integral}
    Let $\sh{L}$ be a Hecke operator such that $\hat{\sh{L}}_{(X)}\in\Z_{\{q\}}[\delta_{(X)}^{\frac{1}{2}}A_X^*]^{\shear}$, in the notation of \eqref{eqn:ShiftedHecke}. Let $\phi=\sh{L}\cdot\Phi_0$, then $\phi(x)\in\Z[q^{-1}]$ for all $x\in X$.
\end{lemma}
\begin{proof}
  Under the various hypotheses discussed in the previous section, the second term in Theorem~\ref{thm:InverseSatake} and Conjecture~\ref{conj:InverseSatake} lies in the ring $\Z_{\{q\}}[[\delta^{\frac{1}{2}}_{(X)}A_X^*]]^\shear$. It follows that
  \begin{align*}
    \phi(x_{\check{\lambda}})&\in\delta_{P(X)}^{-\frac{1}{2}}(x_{\check\lambda})\delta_{(X)}^{\frac{1}{2}}(x_{\check\lambda})q^{\bra\check{\lambda},\rho\ket}\Z_{\{q\}}\\
    &=q^{\bra\check\lambda,-\rho_{P(X)}+\rho_{(X)}+\rho\ket}\Z_{\{q\}}
  \end{align*}
  Since $\rho=\rho_{P(X)}+\rho_{(X)}$ and $2\rho_{P(X)}$ is integral, this is an integral power of $q$.
\end{proof} 

\subsection{Divisibility properties}
For Euler system applications, we are interested in proving that for certain $x$, the values $\phi(x)$ are in fact divisible by $q-1$. By the previous calculations, the values $\phi(x)$ are polynomials in $\Z[q^{\pm 1}]$. Now treat $q$ as a formal variable. To prove divisibility by $q-1$, it remains to specialize at $q=1$ and prove that the value of this polynomial is 0. We will do this by an anti-symmetry argument.

\subsubsection{Strongly tempered case}
Let $X$ be a strongly tempered spherical variety. In \cite[\S 4.3, 4.4]{BZSV}, the authors attached a symplectic $\check{G}$-representation to $X$
\[
  S_X=S_X^\circ\oplus S_X^\bullet,\quad S_X^\bullet\simeq(S_X^\circ)^\ast
\]
By \cite[Proposition 9.3.3]{BZSV}, the multiset of weights of $S_X$ exactly agrees with $\Theta_X:=\Theta_X^+\sqcup(-\Theta_X^+)$ under the assumptions of Theorem~\ref{thm:InverseSatake} or Conjecture~\ref{conj:InverseSatake}. While we are citing general results, in each individual case, the fact that $\Theta$ is exactly the multiset of weights of a symplectic representation can be checked explicitly. Moreover, the assumptions also imply that the weights appearing in $S_X$ are all minuscule.

According to the polarization, we have another decomposition $\Theta=\Theta^\circ\sqcup\Theta^\bullet$. Let $\Theta^{\circ,+}=\Theta^\circ\cap\Theta^+$. This is the multiset of positive weights occurring in $S_X^\circ$, where positive means lying in the cone $\sh{C}_X$ defined in Remark~\ref{rmk:Cone}. Note that we again have $\Theta^\circ=\Theta^{\circ,+}\sqcup(-\Theta^{\circ,+})$.

Let
\[
  \hat{\sh{L}}=\prod_{\check{\theta}\in\Theta_X^\circ}(1-q^{-\frac{1}{2}}e^{\check{\theta}})\in\C[\Lambda]
\]
By construction, this is $W$-invariant, so it is the image of a Hecke operator $\sh{L}\in\sh{H}(G)$. Since $\Theta_X$ consists of minuscule weights, this lies in the integral ring $\Z_{\{q\}}[A^*]^\shear$. If $S_X^\circ$ is irreducible, then $\sh{L}$ is exactly the Hecke polynomial attached to the representation $S_X^\circ$ evaluated at $q^{-\frac{1}{2}}$.
\begin{prop}\label{prop:DivisibleT}
  Let $\sh{L}$ be the Hecke operator defined above, and let $\phi=\sh{L}\cdot\Phi_0$. If $x$ lies on a wall of type T, then $\phi(x)\in (q-1)\Z[q^{-1}]$.
\end{prop}
\begin{proof}
  By the inverse Satake transform, the value $\phi(x_{\check\lambda})|_{q=1}$ is the coefficient of $e^{\check\lambda}$ in the power series expansion of
  \[
    \hat{\sh{L}}|_{q=1}\cdot\frac{\mathtt{ad}}{\mathtt{L}},\quad\text{where }\mathtt{ad}=\prod_{\check\gamma\in\check\Phi_X^+}(1-e^{\check\gamma}),\quad\mathtt{L}=\prod_{\check\theta\in\Theta_X^+}(1-e^{\check\theta}).
  \]
  For our choice of $\sh{L}$, this expression simplifies to
  \begin{align*}
    \mathtt{ad}\cdot\frac{\prod_{\check{\theta}\in\Theta_X^\circ}(1-e^{\check{\theta}})}{\prod_{\check\theta\in\Theta_X^+}(1-e^{\check\theta})}&=\mathtt{ad}\cdot\prod_{\check{\theta}\in\Theta_X^{\circ,+}}\frac{1-e^{-\check{\theta}}}{1-e^{\check{\theta}}}\\
    &=\mathtt{ad}\cdot\prod_{\check{\theta}\in\Theta_X^{\circ,+}}(-e^{-\check{\theta}})
  \end{align*}
  This is a \emph{polynomial} in $\C[\Lambda_X]$, so the subtlety explained in Remark~\ref{rmk:Cone} no longer applies. Denote the second term above by $\mathtt{L}'$.
  
  Suppose $\gamma$ is a spherical root of type T, with corresponding coroot $\check\gamma$. Let $w_\gamma\in W_X$ be its associated simple reflection. We consider the behaviour of the above expression under $w_\gamma$. There are two terms:
  \begin{enumerate}
    \item Since $\Phi_X^+$ is the positive part of a root system, we see that
    \[
      \frac{w_\gamma\mathtt{ad}}{\mathtt{ad}}=\frac{1-e^{-\check\gamma}}{1-e^{\check\gamma}}=-e^{-\check\gamma}
    \]
    \item We need to understand the difference between $w_\gamma\Theta_X^+$ and $\Theta_X^+$. In the homogeneous case, this is described in \cite[Statement 7.1.5]{SakSpherical}. Since $\gamma$ has type T, there are two virtual colors $D$ and $D'$ belonging to $\gamma$. Their associated valuations $v_D$ and $v_{D'}$ sum to $\check{\gamma}$, so
    \[
      \frac{w_\gamma\mathtt{L}'}{\mathtt{L}'}=e^{-(v_D+v_{D'})}=e^{-\check{\gamma}}
    \]
    In general, the above discussion still holds due to our minuscule hypothesis combined with \cite[Theorem 7.1.9(iii)]{SakWang}.
  \end{enumerate}
  It follows that
  \begin{equation}\label{eqn:sign1}
    w_\gamma\left(\hat{\phi}|_{q=1}\cdot\frac{\mathtt{ad}}{\mathtt{L}}\right)=-\hat{\phi}|_{q=1}\cdot\frac{\mathtt{ad}}{\mathtt{L}}
  \end{equation}
  If $w_\gamma\check{\lambda}=\check{\lambda}$, or equivalently if $\check{\lambda}$ lies on the wall defined by $\gamma$, then we see that the coefficient of $e^{\check{\lambda}}$ must be 0, as required.
\end{proof}

\begin{remark}\label{rmk:TypeG}
  Performing the above computation for a spherical root of type G recovers equation~\eqref{eqn:sign1}, except with a plus sign. Therefore, we gain no information. In particular, all roots are of type G in the group case $\bH\backslash\bH\times\bH$, so we do not see any automatic divisibility in the classical Satake isomorphism.
\end{remark}

\begin{remark}
  It is important that we reduced to a polynomial first before considering the action of the Weyl group, since a reflection $w_\gamma$ changes the cone $\sh{C}_X$ that we are expanding along, and there is no \emph{a priori} relation between the coefficients of the two expansions. Note that this cancellation happened exactly because of our choice of $\sh{L}$. This fits nicely with the optimality results of A.~Groutides, cf.~\cite{GroutidesRS}.
\end{remark}

\subsubsection{Non-strongly tempered case --- an example}\label{ss:NonST}
In the general case, the structure of the denominator cancellation is more subtle due to the extra term in \cite[(4.16)]{BZSV}. Since we only have one application in mind which is not strongly tempered, namely the Friedberg--Jacquet setting, we will write down everything explicitly and observe the cancellation ``by hand''. This is also an opportunity to give an example.

In this subsection, we specialize to the following setting:
\[
  \bH=\GL_n\times\GL_n\hookrightarrow \bG=\GL_{2n}\times\G_m,\quad (h_1,h_2)\mapsto\Big(\diag(h_1,h_2), \frac{\det h_1}{\det h_2}\Big)\]
The dual group in this case is $\check{G}_X=\mathrm{Sp}_{2n}\times\G_m$, so it is not strongly tempered. Note that $\bG$ contains the torus factor, so in later sections it will be denoted by $\widetilde{\bG}$. We avoid the tilde here to simplify notations.

In a suitable coordinate, we have
\[
  \check{\Phi}_X^+=\{\epsilon_i\pm\epsilon_j\,|\,1\leq i<j\leq n\}\cup\{2\epsilon_i\,|\,1\leq i\leq n\}
\]
The spherical coroots consist of $\{\epsilon_i-\epsilon_{i+1}\,|\,1\leq i<n\}$ and $2\epsilon_n$. Only the final one is of type T.

Following the recipe of \cite[\S 7.1]{SakSpherical}, one can compute that
\[
  \Theta_X^+=\{(\epsilon_i\pm\epsilon_j,2)\,|\,1\leq i<j\leq n\}\cup\{(\epsilon_i\pm s,1)\,|\,1\leq i\leq n\},
\]
where the notation $s$ for the cocharacter on $\G_m$ is chosen in reference to its ultimate role as the $s$-variable in the $L$-function. In fact, in this interpretation, the first set corresponds to $L(1,\wedge^2,\pi)$, and the second term corresponds to the product $L(\frac{1}{2}+s,\pi)L(\frac{1}{2}-s,\pi)$.

In the above coordinates, the Hecke operator giving rise to the standard degree $2n$ $L$-function has the Satake transform
\[
  \hat{\sh{L}}=\prod_{i=1}^n\Big(1-q^{-\frac{1}{2}}e^{-\epsilon_i+s}\Big)\Big(1-q^{-\frac{1}{2}}e^{\epsilon_i+s}\Big)
\]
We now prove exactly the same statement as Proposition~\ref{prop:DivisibleT}.
\begin{prop}\label{prop:DivisibleFJ}
  Let $\sh{L}$ be the Hecke operator defined above, and let $\phi=\sh{L}\cdot\Phi_0$. If $x$ lies on a wall of type T, then $\phi(x)\in (q-1)\Z[q^{-1}]$.
\end{prop}
\begin{proof}
  As before, the value $\phi(x_{\check{\lambda}})|_{q=1}$ is given by the coefficient of $e^{\check{\lambda}}$ in the power series expansion of the expression
  \[
    \hat{\sh{L}}|_{q=1}\cdot\frac{\prod_{\check{\gamma}\in\check{\Phi}_X^+}(1-e^{\check{\gamma}})}{\prod_{\check{\theta}\in\Theta_X^+}(1-e^{\check{\theta}})}.
  \]
  By comparing the above descriptions of the multisets $\check{\Phi}_X^+$ and $\Theta_X^+$, we see that most terms of the fraction cancels, and it simplifies to
  \begin{align*}
    \prod_{i=1}^n(1-e^{-\epsilon_i+s})(1-e^{\epsilon_i+s})\cdot\prod_{i=1}^n\frac{1-e^{2\epsilon_i}}{(1-e^{\epsilon_i+s})(1-e^{\epsilon_i-s})}&=\prod_{i=1}^n (1-e^{2\epsilon_i})\cdot\frac{1-e^{-\epsilon_i+s}}{1-e^{\epsilon_i-s}}\\
    &=\prod_{i=1}^n(1-e^{2\epsilon_i})(-e^{-\epsilon_i+s})
  \end{align*}
  The only root of type T corresponds to the reflection $\epsilon_n\mapsto -\epsilon_n$. This clearly multiplies the final expression by $-1$. The proposition follows.
\end{proof}

\subsubsection{Non-split case --- an example}\label{ss:NonSplit}
When we study the twisted Friedberg--Jacquet setting, we will also encounter a case where $\bG$ is not split over $F$. This is not covered by \cite{SakSpherical}, but the necessary inversion formula can be computed using the same method(See the Appendix for details). The result strongly resembles Theorem~\ref{thm:InverseSatake}, and we will define the combinatorial data appearing in it by analogy.

Let $E/F$ be the unramified quadratic extension with associated quadratic character $\eta$ and non-trivial involution $\bar{\cdot}$. Let $w\in \GL_n(E)$ be the anti-diagonal matrix with entities $1$. Define
\[
    \U_{2n}=\{g\in \GL_{2n}(E)| \bar{g}^T\begin{pmatrix}0 & w \\ w & 0\end{pmatrix}g= \begin{pmatrix}0 & w \\ w & 0\end{pmatrix}\}.
\]
This is the quasi-split unitary group in $2n$ variables. Consider the following embedding
\[
    \mb{H}=\Res_{E/F}\GL_n\hookrightarrow\bG =\U_{2n}\times \G_m \quad h\mapsto (\diag(h, w \bar{h}^{T}w),\mathrm{Norm}_F^E\det h)
\]
As in the previous subsection, the group $\bG$ here will play the role of $\widetilde{\bG}$ in later sections. The group $\mb{G}$ is not split, and we recall that in this case, the Satake isomorphism is
\[
    C_c^\infty(K\backslash G/K,\C)\cong \C[q^{\pm 2 z_1},\cdots, q^{\pm 2z_n},q^{\pm s}]^W,
\]
where $W$ is the Weyl group for the root system of type $\mathrm{C}_n$.

Set $\mb{X}:=\bG/\mb{H}$. This is a form of the variety $\mb{X}$ from the previous subsection. Let
\[
    \Lambda_X^+=\{(\lambda_1,\cdots,\lambda_n)\in\Z^n\,|\,\lambda_1\geq\cdots\geq\lambda_n\}.
\]
Then we have the generalized Cartan decomposition
\[
    X=\bigsqcup_{\check{\lambda}\in\Lambda_X^+}x_{\check{\lambda}}K,
\]
which is analogous to Proposition~\ref{prop:GeneralCartan}. We define the spherical root system to have positive coroots
\[
  \check{\Phi}_X^+=\{\epsilon_i\pm\epsilon_j\,|\,1\leq i<j\leq n\}\cup\{2\epsilon_i\,|\,1\leq i\leq n\}.
\]
It is a root system of $\mathrm{C}_n$. Among the (simple) spherical roots, we call $2\epsilon_n$ the only root of type T. This is still in agreement with the previous subsection when base changed to $E$.

The analogue of $\Theta_X^+$ now contains the additional data of signs $\sigma_{\check{\theta}}\in\{\pm 1\}$. This was also present in \cite{SakSpherical} to account for non-split $\mb{H}$. Define
\[
  \Theta_X^+=\{(\epsilon_i\pm\epsilon_j,-,2)\,|\,1\leq i<j\leq n\}\cup\{(\epsilon_i\pm 2s,-,2)\,|\,1\leq i\leq n\}.
\]
Then the analogue of Theorems~\ref{thm:RelSatake} and \ref{thm:InverseSatake} in our case is the following result, whose proof will be given in the appendix (Theorem \ref{RSUFJ}). 

\begin{theorem}\label{thm:Satake non-split FJ}
\begin{enumerate}
    \item There exists an isomorphism of $C_c^\infty(K\backslash G/K,\C)$-modules 
\[
    C_c^\infty(X,\C)^K\cong \C[q^{\pm z_1},\cdots, q^{\pm z_n},q^{\pm s}]^W
\]
which is equivariant with respect to $C_c^\infty(K\backslash G/K,\C)$ acting by convolution on the left and by multiplication via the Satake isomorphism on the right.
    \item The isomorphism sends the basic function $\indf[\mb{X}(\Or)]$ to the constant function 1.
    \item There is an inversion formula
  \[
        \delta_{P(X)}^{\frac{1}{2}}\cdot(\sh{L}\cdot\Phi_0) = \hat{\sh{L}}\cdot\frac{\prod_{\check{\gamma}\in\check{\Phi}_X^+}(1-e^{\check{\gamma}})}{\prod_{(\check{\theta},r_{\check{\theta}}, d_{\check\theta})\in\Theta_X^+}(1-r_{\check{\theta}}q^{-\frac{1}{2}d_{\check{\theta}}}e^{\check{\theta}})}\bigg|_{\Lambda_X^+}
    \]
with \[\delta_{P(X)}^{1/2}(x_\lambda)=
    (-1)^{\sum_i \lambda_i(n-i+1)}q^{\sum_i \lambda_i(n-i+1/2)}\]
\end{enumerate}
\end{theorem}

The Hecke operator giving rise to the degree $4n$ $L$-function has the Satake transform
\begin{align*}
  \hat{\sh{L}}_{\mathrm{BC}}&=\prod_{i=1}^n\Big(1-qe^{-2\epsilon_i+2s}\Big)\Big(1-qe^{2\epsilon_i+2s}\Big)\\
  &=\prod_{i=1}^n\Big(1-q^{1/2}e^{-\epsilon_i+s}\Big)\Big(1-q^{-\frac{1}{2}}e^{\epsilon_i+s}\Big)\prod_{i=1}^n\Big(1+q^{1/2}e^{-\epsilon_i+s}\Big)\Big(1+q^{-\frac{1}{2}}e^{\epsilon_i+s}\Big)
\end{align*}
According to this decomposition, write $\hat{\sh{L}}_{\mathrm{BC}}=\hat{\sh{L}}_\emptyset\cdot\hat{\sh{L}}_\eta$
As in Proposition \ref{prop:DivisibleFJ}, one finds that  \[
    \hat{\sh{L}}_\eta|_{q=1}\cdot\frac{\prod_{\check{\gamma}\in\check{\Phi}_X^+}(1-e^{\check{\gamma}})}{\prod_{\check{\theta}\in\Theta_X^+}(1-e^{\check{\theta}})}.
  \] simplifies to  \begin{align*}
    \prod_{i=1}^n(1+e^{-\epsilon_i+s})(1+e^{\epsilon_i+s})\cdot\prod_{i=1}^n\frac{1-e^{2\epsilon_i}}{(1+e^{\epsilon_i+s})(1+e^{\epsilon_i-s})}
    =\prod_{i=1}^n(1-e^{2\epsilon_i})(e^{-\epsilon_i+s})
  \end{align*}
  From this, exactly the same anti-symmetry argument from before gives the following result.
  \begin{prop}\label{prop:Divisible non-split FJ}
  Let $\sh{L}$ be the Hecke operator whose Satake transform is $\hat{\sh{L}}_{\mathrm{BC}}$, and let $\phi=\sh{L}\cdot\Phi_0$. If $x$ lies on a wall of type T, then $\phi(x)\in (q-1)\Z[q^{-1}]$.
\end{prop}
\subsection{Abstract Euler system I}
We now combine the results from the previous subsections together to produce an abstract tame norm relation.

\begin{definition}\label{def:Hecke}
  In the set-up of Propositions~\ref{prop:DivisibleT}, \ref{prop:DivisibleFJ}, or \ref{prop:Divisible non-split FJ}, we say that $\sh{L}$ is \emph{the Hecke operator attached to} $\mb{X}$. Recall that if $\mb{X}$ is strongly tempered, then $\sh{L}$ is the denominator of the unramified Plancherel formula for $\mb{X}$.
\end{definition}

In the rest of this article, we tacitly assume that when this terminology is applied, the relevant pair $(\mb{G},\mb{X})$ satisfies the conditions in one of the propositions listed. This definition is rather \emph{ad hoc} beyond the strongly tempered case. One candidate for the general definition is the piece corresponding to $S_X$ in the notation of \cite[\S 9.1]{BZSV}, but we will not pursue this further in this paper.

%
%

\begin{prop}\label{prop:MainCorollary}
  Let $\tilde{\mb{X}}\to\mb{X}$ be a combinatorially trivial $\G_m$-bundle. Let $\sh{L}$ be the Hecke operator attached to $\tilde{\mb{X}}$, then there exists a function
  \[
    \Phi_1\in C_c^\infty(\tX,\Z[q^{-1}])^{K^1}
  \]
  such that
  \[
    \Tr_{K^0}^{K^1}\Phi_1=\sh{L}\cdot\indf[\widetilde{\mb{X}}(\Or)],
  \]
  where $\Tr_{K^0}^{K^1}:=\sum_{g\in K^0/K_1}g$ is the trace operator.
\end{prop}
\begin{proof}
  Let $\phi\in C_c^\infty(\tX,\C)^{K^1}$. Let $x\in X$. Write $xK^0=\bigsqcup_{i\in I}x_i K^1$, then
  \[
    (\Tr_{K^0}^{K^1}\phi)(x)=\sum_{i\in I}\frac{[K^0:K^1]}{[xK^0:x_iK^1]}\phi(x_i).
  \]
  Since $K^1$ is normal in $K^0$, the multiplier in front of each term is the same. In other words, the image of $C_c^\infty(\tX,\Z[q^{-1}])^{K^1}$ under the trace map is characterized by the divisibility condition
  \[
    \phi(x)\in\frac{q-1}{[xK^0:xK^1]}\Z[q^{-1}]\text{ for every }x
  \]
  By Proposition~\ref{prop:Volume} (applied after a quadratic extension if $\mb{G}$ is not split), if this coefficient is not 1, then $x$ lies on a wall of type T. The three propositions cited in the statement applied to $\widetilde{\mb{X}}$ shows that this divisibility requirement is satisfied.
\end{proof}

\begin{remark}\label{rmk:1}
\begin{enumerate}
    \item In all cases, the Hecke operator defined gives rise to the $L$-value attached to the spherical variety $\mathbf{X}$, in the sense of \cite[Definition 7.2.3]{SakSpherical}. This lends credence to the idea the Euler systems, including the tame norm relations, should the ``arithmetic shape of zeta functions'' \cite{KatoEulerSystem}. It also suggests further subtleties in non-wavefront cases, such as the one investigated by Cornut \cite{CornutES}.
    \item For any Hecke operator $\sh{H}\in C_c^\infty(K\backslash G/K,\Z[q^{-1}])$, the result of this proposition holds for the convolution $\sh{L}\ast\sh{H}$. This is easily verified directly on the automorphic side, or one can work on the spectral side and observe that the same ingredients, namely denominator cancellation and Weyl symmetry, still hold. This fits with the results of \cite{GroutidesRS}.
\end{enumerate}
\end{remark}

\section{Construction of Euler systems}
Now let $F$ be a global field. The goal of this section is to apply the above theorems and conjectures to construct the tame part of Euler systems over $F$.

\subsection{Motivic theta series}
Let $\A$ be the ring of ad\`{e}les of $F$. Let $p$ be a rational prime. Let $S$ be a finite set of places of $F$ such that
\begin{enumerate}
  \item $S$ contains all places above $p$ and $\infty$.
  \item Away from $S$, the assumptions made in \S\ref{ss:LocalAssumption} hold \cite[Proposition 2.3.5]{SakANT}.
\end{enumerate}
In particular, we may extend $\mb{X}$ and $\mb{G}$ to smooth varieties over $\spec\Or_{F,S}$, which allows us to talk about their ad\`{e}lic points. Again denote these extensions by the same letters. 

Let $K^S=\prod_{v\notin S}\mb{G}(\Or_{F_v})$. Define the \emph{basic vector} to be
\[
  \Phi_0:=\prod_{v\notin S}\indf[\mb{X}(\Or_{F_v})]\in C_c^\infty(\mb{X}(\A^{S}),\Z_p)
\]
By construction, it is $K^S$-invariant.

\begin{definition}\label{defn:MotivicTheta}
  Let $\sh{M}$ be a $\Z_p$-module with a smooth $\bG(\A^S)$-action. An $\mb{X}$-\emph{theta series} for $\sh{M}$ is a $\bG(\A^S)$-equivariant map
  \[
    \Theta:C_c^\infty(\mb{X}(\A^S),\Z_p)\to\sh{M}
  \]
  The basic element $z\in\sh{M}$ is defined to be $\Theta(\Phi_0)$.
\end{definition}
\begin{remark}\label{rmk:FunctionSpace}
    In general, the discussions of \cite[\S 3]{SakANT} suggests that the correct function space should be a restricted tensor product of smooth functions on $\mb{X}^\bullet$ with bounded growth near the boundary $\mb{X}-\mb{X}^\bullet$. However, this subspace is sufficient for our applications.
\end{remark}

If $\sh{M}$ is a space of automorphic functions, then the classical theta series is an example of this definition. In the relative Langlands program, such objects are expected to play an important role, cf.~\cite[\S 15.1]{BZSV}. Their motivic analogues should also exist in great generality, though they are harder to construct. We will later specify $\sh{M}$ to be certain $\Z_p$-coefficient continuous \'etale cohomology groups, which we consider as $p$-adic realizations of motives.

\begin{example}\label{ex:SiegelUnits}
    Consider the case $\mb{G}=\GL_2$, $F=\Q$, and $\mb{X}=\mathtt{std}$, the 2-dimensional affine space with the standard action of $\mb{G}$. The theory of Siegel units (cf.~\cite[\S 1.4]{KatoEulerSystem}, \cite[Th\'eor\'eme 1.8]{Colmezp-BSD}) gives a motivic theta series 
    \[
        C_c^\infty(\mb{X}(\A^S),\Z)\to\varinjlim_U\Or(Y(U))^\times,
    \]
    where $U$ runs over all open compact subgroups of $\mb{G}(\A^\infty)$, and $Y(U)$ is the open modular curve of level $U$. By taking $p$-adic \'etale realization, this produces a motivic theta series valued in $\h^1_\cont(Y,\Z_p(1))$. This is exactly the construction of $\GL_2$-Eisenstein classes in weight 2. For higher weights, the coefficient $\Z_p$ is replaced by an integral \'etale local system, cf.~\cite[Proposition 7.2.4]{LSZ}.

    We are hiding one subtlety in this definition. In the above references, one need to restrict to functions which vanish at $(0,0)$ to ensure convergence. To get this vanishing automatically, we fix an auxiliary place $v\in S$ and fix the test vector there to vanish at $(0,0)$. This was used in the choices of \cite[\S 8.4.4]{LSZ}. 
\end{example}

\begin{example}
    By taking cup product, we obtain the following motivic theta series for $\mb{X}=\mathtt{std}^2$ with the diagonal $\GL_2$-action
    \[
        C_c^\infty(\mb{X}(\A^S),\Z)\to\h^2_\cont(Y,\Z_p(2)).
    \]
    This is Kato's construction of his zeta element, reformulated in \cite{Colmezp-BSD}. As observed by Colmez, this makes verifying their norm relations ``almost automatic'', since it reduces to some elementary manipulation of indicator functions on $\mb{X}$.
\end{example}

\subsection{Abstract Euler system II}
Let $\mb{T}$ be a 1-dimensional torus defined over $\Or_{F,S}$, the $S$-integers of $F$. Let $\widetilde{\mb{X}}\to\mb{X}$ be a combinatorially trivial $\mb{T}$-bundle. Let $\widetilde{\bG}=\bG\times\mb{T}$ be the augmented group. For each square-free ideal $\m$ not divisible by any place in $S$, we may form the level structure
\begin{equation}\label{eqn:J[m]}
  J[\m]:=\prod_{\substack{v\nmid\m\\v\notin S}}\mb{T}(\Or_{F_v})\prod_{v|\m}\mb{T}(\Or_{F_v})^1\subseteq\mb{T}(\A^S),\quad K[\m]:=K^S\times J[\m]\subseteq\widetilde{\bG}(\A^S),
\end{equation}
where $\mb{T}(\Or_{F_v})^1$ consists of the elements of $\mb{T}(\Or_{F_v})$ whose reduction is the identity.

Let $\mathscr{L}$ be the set of places of $F$ away from $S$ where the hypotheses of Section~\ref{sec:Spherical} are satisfied for $\widetilde{\mb{X}}$. This in particular requires $\mb{T}$ to be split at those places. For each $\ell\in\mathscr{L}$, let $\sh{H}_\ell$ be the Hecke operator attached to $\mb{X}$ (Definition~\ref{def:Hecke}). They give rise to the following abstract Euler system.

\begin{theorem}\label{thm:Abstract}
  Let $\Theta$ be an $\widetilde{\mb{X}}$-theta element for $\sh{M}$. Let $\mathscr{R}$ be the collection of square-free products of places in $\mathscr{L}$. There exists a collection of elements $\{z_\m\in\sh{M}^{K[\m]}\,|\,\m\in\mathscr{R}\}$ such that $z_1$ is the basic element, and for any $\m\in\mathscr{R}$ and any $\ell\in\mathscr{L}$ such that $\ell\nmid\m$, we have the norm relation
  \[
    \Tr_{K[\m]}^{K[\m\ell]}z_{\m\ell}=\sh{H}_\ell\cdot z_\m
  \]
\end{theorem}
\begin{proof}
  For each $\ell\in\mathscr{L}$, let $\Phi^{(\ell)}$ be the element $\Phi_1$ from Proposition~\ref{prop:MainCorollary} applied to the completion $F_\ell$. Define $\Phi[\m]$ the same way as $\Phi_0$, except the local component at each $\ell|\m$ are replaced by $\Phi^{(\ell)}$. The theorem follows from the $\widetilde{\bG}(\A^\infty)$-equivariance of $\Theta$.
\end{proof}

\subsection{\'Etale classes and Galois action}\label{ss:MotivicES}
We now specialize further and suppose that $\mb{G}$ has Shimura varieties defined over a number field $E$ containing $F$. We can then form Shimura varieties for $\widetilde{\bG}$. Let $\mathbb{L}$ be a $\Z_p$-local system defined using an algebraic representation of $\mb{G}$, as in for example \cite[\S III.2]{HT2001}. Define
\begin{equation}\label{eqn:Motivic}
  \sh{M}=\varinjlim_{K^S}\h^d_\cont(\Sh_{\widetilde{\bG}}(K_SK^S),\mathbb{L}),
\end{equation}
where $K_S$ is a fixed open compact subgroup of $\widetilde{\bG}(\A_S)$, and $K^S$ runs over all open compact subgroups of $\widetilde{\bG}(\A^S)$. The cohomology theory used is Jannsen's continuous \'etale cohomology \cite{JannsenContEt}. Note that all the transition maps are well-defined since we are not changing the level structure at the $p$-adic places.

Suppose $K\subseteq\bG(\A^\infty)$ and $U\subseteq\mb{T}(\A^\infty)$ are both open compact, then we have
\begin{equation}\label{eqn:FieldExtn}
  \Sh_{\widetilde{\bG}}(K\times U)\simeq\Sh_{\bG}(K)\times\Sh_{\mb{T}}(U)
\end{equation}
as algebraic varieties over $E$. Take $U=J[\m]\times J_S$ for some appropriate choice of the ramified level structure $J_S$, then it is often possible to interpret the above product as a base change from $E$ to an extension $E[\m]$. We single out two cases.
\begin{itemize}
  \item If $\mb{T}=\G_m$ with Shimura cocharacter $z^{-1}$, then $E[\m]$ is the cyclotomic extension of conductor $\m$.
  \item If $\mb{T}=\U(1)$ attached to a quadratic extension $E/F$, and the Shimura cocharacter is $z\mapsto \bar{z}/z$, then $E[\m]$ is the anticyclotomic extension of $E$ of conductor $\m$.
\end{itemize}
In these cases, the trace from Theorem~\ref{thm:Abstract} can be identified with a trace map between fields.

Locally at $\ell\in\mathscr{L}$, we have the following description of the Hecke algebra for $\widetilde{\bG}$ as a ring of Laurent polynomials
\begin{equation}\label{eq:HeckePoly}
    \sh{H}(\widetilde{\bG}(F_\ell),\C)\simeq \sh{H}(\bG(F_\ell),\C)[\mathtt{T}^{\pm 1}].
\end{equation}
Here, $\mathtt{T}$ is the indicator function of $\varpi\Or_{F_\ell}^\times\subseteq\mb{T}(F_\ell)$.

Under the identification \eqref{eqn:FieldExtn}, the action of $\mathtt{T}$ matches with the action of the geometric Frobenius element $\mathrm{Frob}_\ell^{-1}$. In the $\mb{T}=\U(1)$ case, this requires further comments: the identification of $\mb{T}_{/F_\ell}$ with $\G_m$ depends on the choice of a place $\lambda$ of $E$ above $\ell$, and $\mathrm{Frob}_\ell$ really means the arithmetic Frobenius at $\lambda$. Using this, Theorem~\ref{thm:Abstract} leads to the following Corollary.
\begin{cor}\label{cor:Abstract}
  Let the setting be as in Theorem~\ref{thm:Abstract}, with $\sh{M}$ given by equation~\eqref{eqn:Motivic}. Let $\sh{H}_V$ be the Hecke operator attached to $\widetilde{\mb{X}}$, viewed as a Hecke polynomial via \eqref{eq:HeckePoly}. Then there exists a collection of elements
  \[
    \big\{z_\m\in\h^d_\cont(\Sh_{\bG/E[\m]},\mathbb{L})\,\big|\,\m\in\mathscr{R}\big\}
  \]
  satisfying the norm relations
  \[
    \Tr_{E[\m]}^{E[\m\ell]}z_{\m\ell}=\sh{H}_V(\mathrm{Frob}_\ell^{-1})\cdot z_\m
  \]
  whenever $\m,\m\ell\in\mathscr{R}$. Here, $\Fr_\ell^{-1}$ acts by its image in $\Gal(E[\m]/E)$.
\end{cor}

\begin{remark}
  It is crucial that the Shimura cocharacter on $\mb{T}$ is non-trivial, otherwise we do not recover a Galois action. In particular, choosing the trivial equivariant bundle $\widetilde{\mb{X}}=\mb{X}\times\mb{T}$ does not lead to interesting arithmetic in any of the examples we consider, as expected.
\end{remark}

This result of Corollary~\ref{cor:Abstract} may be called a ``motivic Euler system'', cf.~\cite[\S 9.4]{LSZ-U3}. We now explain the steps required to convert this to an Euler system in Galois cohomology, proving Corollary~\ref{cor:ActualES} from the introduction. These steps are all standard in the literature.
\begin{enumerate}
  \item Null-homologous modification: the main result of \cite{MorelSuh} constructs an element $\mathtt{t}^{\pm}$ in the Hecke algebra of $\mb{G}$ whose action on the cohomology of $\Sh_{\bG}$ are the two sign projectors. In particular, $\mathtt{t}:=\mathtt{t}^{(-1)^{d+1}}$ annihilates the cohomology group $\h^d(\Sh_{\bG/\bar{E}},\mathbb{L})$. Therefore, for each $\m$ as above, $\mathtt{t}\cdot z_\m$ is cohomologically trivial, and the spectral sequence in continuous \'etale cohomology \cite[Remark 3.5(b)]{JannsenContEt} gives an element
  \[
    \tilde{c}_\m\in\h^1(E[\m],\h^{d-1}(\Sh_{\bG/\bar{E}},\mathbb{L}))
  \]
  A simpler argument is given in \cite[Proposition 6.9(1)]{LL2021} which also suffices for our purpose.
  \item Projection to Galois representations: Kottwitz's conjecture gives a concrete description of the cohomology group $\h^{d-1}(\Sh_{\bG/\bar{E}},\mathbb{L}\otimes_{\Z_p}\Q_p)$. In the case considered in Corollary~\ref{cor:ActualES}, it gives a projection from this group to the Galois representation $\rho_\pi$. Moreover, in the cases considered below, the Hecke polynomial specialized to the Satake parameters of $\pi$ agrees with the characteristic polynomial of $\rho_\pi$.
  
  The integral structure given by the lattice $\mathbb{L}$ is mapped to a lattice $T_\pi$ in $\rho_\pi$. The image of $\tilde{c}_\m$ under this projection is the class $c_\m\in\h^1(E[\m],T_\pi)$.
\end{enumerate}

In the following two subsections, we will explain how pushforwards of cycles and Eisenstein classes both give instances of such a motivic theta series. In these settings, the above corollary produces many of the known motivic Euler systems in the literature.

\subsection{Special cycles}\label{ss:Diagonal}
Suppose that $\mb{X}$ is homogeneous, then $\mb{X}=\mb{H}\backslash\bG$, where we recall that $\mb{H}$ is the stabilizer of a point in the open orbit. In this case, $\mb{X}$ is automatically smooth, and it is affine if and only if $\mb{H}$ is reductive. Assume the following conditions.
\begin{enumerate}
  \item Both $\mb{G}$ and $\mb{H}$ have Shimura varieties, and $\dim\Sh_{\bG}=2\dim\Sh_{\mb{H}}+1$.
  \item $\xi$ is an algebraic representation of $\mb{G}_\infty$ whose restriction to $\mb{H}_\infty$ contains the trivial representation.
\end{enumerate}
Under the above assumptions, the discussions of \cite[\S 2.4]{LaiSkinner} carries over verbatim, and Proposition 2.7 of \emph{op.~cit.}~gives a weaker version of the motivic theta series.
\begin{prop}\label{prop:cycle}
  Let $d=\dim\Sh_{\mb{H}}+1$. Let $\mathbb{L}$ be the $\Z_p$-local system on $\Sh_{\bG}$ attached to $\xi$. The diagonal cycle construction defines a map
  \[
    \Theta:C_c^\infty(\mb{H}(\A^{p\infty})\backslash\mb{G}(\A^{p\infty}),\Z_p)\to\h_\cont^{2d}(\Sh_{\bG},\mathbb{L}(d))
  \]
\end{prop}

\begin{remark}
  \begin{enumerate}
    \item The reason this is weaker than a motivic theta series is that for a general ring $R$, it is \emph{not} true that
    \[
      \mb{H}(R)\backslash\mb{G}(R)=\mb{X}(R).
    \]
    This holds in the Gan--Gross--Prasad case below, but it fails for the unitary Friedberg--Jacquet case. This failure is related to the issue of stabilizing the relative trace formula.
    
    However, this is enough for the purpose of constructing Euler systems. In fact, our main local result Proposition~\ref{prop:MainCorollary} gives an abstract tame norm relation for each $\bG(F_\ell)$-orbit of $\mb{X}(F_\ell)$.
    \item The algebraic representation $\xi$ determines the archimedean part of the representations distinguished by $\Theta$. The above construction depends on the choice of an invariant vector in $\xi|_{\mb{H}}$. It is likely that including the archimedean place in the space $\mb{X}(\A^{p\infty})$ would give a more canonical construction.
  \end{enumerate}
\end{remark}

Since we are constructing motivic classes using cycles, we should limit our deformations to the self-dual setting. Therefore, it seems necessary to take $\mb{T}=\U(1)$, defined with respect to a CM extension $E/F$. To satisfy the conditions in \S\ref{sec:Spherical}, we must restrict our attention to split primes. The result is a split anticyclotomic Euler system, in the sense of \cite{JNS}.

We now explain the content of Corollary~\ref{cor:Abstract} in a few cases. In each case, we need to specify the data of an embedding of reductive groups $\iota:\mb{H}\hookrightarrow\mb{G}$ and a character $\nu:\mb{H}\to\U(1)$. The latter is used to specify the $\U(1)$-bundle over $\mb{X}$. In most cases, the local spherical variety appears in the table \cite[Appendix A]{SakSpherical}, so they have been verified to satisfy the combinatorial conditions alluded to in \S\ref{ss:HomogeneousCase}. The exception is the twisted Friedberg--Jacquet case, where we derive analogoues of Sakellaridis's results ourselves.

\subsubsection{Gan--Gross--Prasad \cite{LaiSkinner}}
Let $\mathtt{V}_{n}\subseteq\mathtt{V}_{n+1}$ be Hermitian spaces of dimensions $n$, $n+1$ respectively which are nearly definite. Consider the setting
\begin{align*}
  &\mb{H}=\U(\mathtt{V}_n),\quad\bG=\U(\mathtt{V}_n)\times\U(\mathtt{V}_{n+1})\\
  &\iota:h\mapsto (h,\diag(h,1)),\quad \nu:h\mapsto\det h
\end{align*}
The intersection of $\bH$ with a good Borel $\mb{B}$ is trivial, so the bundle is automatically combinatorially trivial. The Hecke operator attached to $\mb{X}$ corresponds to the standard tensor product representation of degree $n(n+1)$, so Corollary~\ref{cor:Abstract} recovers \cite[Proposition 5.4]{LaiSkinner}, bypassing the explicit matrix calculations in \S 4 of \emph{op.~cit.}

\subsubsection{Friedberg--Jacquet \cite{GrahamShah}}
Let $\mathtt{V}_{2n}$ be a Hermitian space with signature $(1,2n-1)$ at one archimedean place and $(0,2n)$ at the other archimedean places. Let $\mathtt{W}$ be a totally definite subspace of dimension $n$, and let $\mathtt{W}^\perp$ be its dual. Consider the setting
\begin{align*}
  &\mb{H}=\U(\mathtt{W}^\perp)\times\U(\mathtt{W}),\quad\bG=\U(\mathtt{V})\\
  &\iota:(a,b)\mapsto \diag(a,b),\quad \nu:(a,b)\mapsto\frac{\det a}{\det b}
\end{align*}
At a split place, the local picture is $\GL_n\times\GL_n\backslash\GL_{2n}$. For an appropriate choice of the Borel subgroup $\mb{B}$, the intersection $\mb{A}\cap\mb{H}$ has the form
\[
  (\diag(z_1,\cdots,z_n),\diag(z_n,\cdots,z_1)),
\]
so the bundle defined by $\nu$ is combinatorially trivial (which explains why $(a,b)\mapsto\det a\det b$ is not the correct choice for $\nu$). In this case, the equivariant bundle is non-trivial even as a line bundle, which explains why it is essential to include the additional $\U(1)$-factor.

The Hecke operator attached to $\mb{X}$ corresponds to the standard representation of degree $2n$, so Corollary~\ref{cor:Abstract} recovers the tame part of the split anticyclotomic Euler system for symplectic representations constructed by Graham and Shah \cite[Proposition 9.27]{GrahamShah} without assumptions on the prime $p$.

\subsubsection{Twisted Friedberg--Jacquet}
We now consider the global setting considered in \cite{LXZ-TwistedFL}, so we will change notations to align with theirs. To orient the reader, our Euler system will be a split anticyclotomic Euler system defined relative to the CM extension $F'/F_0$ in the notations below.

We place ourselves in the setting considered in \S\ref{ss:TwistedFJ} from the introduction, which we briefly summarize. Recall that $E/F_0$ is a biquadratic extension with the following intermediate subfields
\begin{center}
    \begin{tikzpicture}[scale=0.8]
    \node (F0) at (0,0) {$F_0$};
    \node (E0) at (1.5,1.5) {$E_0$};
    \node (F') at (0,1.5) {$F'$};
    \node (F) at (-1.5,1.5) {$F$};
    \node (E) at (0,3) {$E$};

    \draw (F0) -- (E0) node [pos=0.5, below right] {Real, $\tau$} -- (E);
    \draw (F0) -- (F) node [pos=0.5, below left] {CM, $\sigma$} -- (E);
    \draw (F0) -- (F') -- (E);
    \end{tikzpicture}
\end{center}
Moreover, $(B,\ast)$ is a division algebra over $F$ of dimension $(2n)^2$ with an embedding $E\hookrightarrow B$. They give rise to the algebraic groups
\[
    \bH=\Res_{E_0/F_0}\U(B^E)=\{g\in (B^E)^\times| g^*g=1\},\quad \bG=\U(B)=\{g\in B^\times|g^*g=1\}.
\]
Both groups have Shimura varieties with reflex fields (contained in) $F$ and dimensions
\[
    \dim_F \Sh_{\mb{H}}=n-1,\quad \dim_F \Sh_{\mb{G}}=2n-1.
\]
Let $\mb{T}=\U(1)_{F'/F_0}$. The reduced form defines a combinatorially trivial character $\nu:\bH\to\mb{T}$, and we have the augmented embedding
\[
    \mb{H}\hookrightarrow\widetilde{\mb{G}}=\bG\times\mb{T},\quad h\mapsto(h,\nu(h))
\]
After base changing to the algebraic closure, this and the previous untwisted case agree.

Let $\ell$ be a place of $F_0$ which splits in $F'$. Away from finitely many bad places, there are two cases:
\begin{itemize}
    \item $\ell$ splits in $F$: the local calculation is explained in \S\ref{ss:NonST}.
    \item $\ell$ is inert in $F$: the local calculation is explained in \S\ref{ss:NonSplit}
\end{itemize}
Therefore, we can apply Corollary~\ref{cor:Abstract} to obtain a motivic Euler system in this setting. We will now unwind the relevant definitions and describe the resulting Euler system.

Recall the following automorphic context from \cite{LXZ-TwistedFL}. Let $\pi_0$ be a symplectic automorphic representation of $\GL_{2n}(\A_{F_0})$ whose base change to $F$ is cuspidal. Let $\pi$ be the automorphic representation of $\mb{G}(\A_{F_0})$ such that $\pi_0$ and $\pi$ have the same base change to $F$. There is a factorization
\begin{equation}\label{eq:TwistedAuto}
    L(s,\BC_E(\pi_0))=L(s,\BC_{F'}(\pi_0))L(s,\BC_{F'}(\pi_0\otimes\eta_{F/F_0}))
\end{equation}
Stated imprecisely, \cite[Conjecture 1.1]{LXZ-TwistedFL} relates the non-vanishing of the $\mb{H}$-period for $\pi$ to the non-vanishing of \emph{one} of the two $L$-functions on the right hand side.

Let $\rho_\pi$ be the Galois representation attached to $\pi_0$. The Kottwitz conjecture (cf.~\cite[Conjecture 5.2]{BlasiusRogawski}) predicts that $\rho_\pi|_{\Gal_F}$ appears in the $\pi$-isotypic component of the cohomology of $\Sh_{\bG}$. Let $\chi$ be an anticyclotomic Hecke character of $\A_{F'}^\times$. Following the Abel--Jacobi map procedure described in \S\ref{ss:MotivicES}, our motivic Euler system gives rise to a class in
\[
    \h^1(E,\Res^E_F\rho_\pi\otimes\Res^E_{F'}\chi)=\h^1(F',(\Ind_{F'}^E\Res^E_F\rho_\pi)\otimes\chi),
\]
where the equality is a consequence of Shapiro's lemma and the projection formula. We have a decomposition
\[
    \Ind_{F'}^E\Res^E_F\rho_\pi=\rho_{\pi}|_{\Gal_{F'}}\oplus(\rho_\pi\otimes\eta_{F/F_0})|_{\Gal_{F'}}.
\]
Therefore, we have constructed Selmer classes for a decomposable representation, whose pieces exactly correspond to the factorization \eqref{eq:TwistedAuto} on the automorphic side.

The characteristic polynomial in this case corresponds to the local $L$-factor for $L(s,\BC_E(\pi_0))$. This is exactly the Hecke operator used in the non-split case (Proposition~\ref{prop:Divisible non-split FJ}). In the split case, it is the square of the Hecke operator used in Proposition~\ref{prop:DivisibleFJ}, since $\eta_{F/F_0}$ is trivial. This does not change the main result, as explained in Remark~\ref{rmk:1}. Therefore, Corollary~\ref{cor:Abstract} gives a motivic Euler system, and the discussions thereafter gives Theorem~\ref{thm:TwistedFJ} from the introduction. Note that in this case, Kottwitz's conjecture is proven in \cite{KSZ}.

\begin{cor}
    Let $\pi$ be as above. Suppose Kottwitz's conjecture holds for the $\pi$-isotypic part of the cohomology of $\Sh_\bG$, then there exists a split anticyclotomic Euler system for the decomposable representation
    \[
        (\rho_\pi\oplus\rho_\pi\eta_{F/F_0})|_{\Gal_{F'}}
    \]
    whose base class is the $p$-adic \'{e}tale realization of the special cycle corresponding to the embedding $\Sh_{\bH}\hookrightarrow\Sh_\bG$.
\end{cor}

\subsection{Eisenstein classes}\label{ss:Eisenstein}
We now move to a family of non-homogeneous spherical varieties. Suppose $V$ is an affine space with a spherical action of a group $\mb{H}$. Let
\[
  \mb{X}=\mb{G}\times^{\mb{H}} V
\]
In a change of notation from \S\ref{sec:Spherical}, the stabilizer of a generic point is no longer $\mb{H}$. In the Eisenstein class applications below, the stabilizer is a mirabolic subgroup of $\mb{H}$, which gives an explanation for their prominence in \cite{LoefflerSpherical}.

We expect there to be a general pushforward construction sending motivic theta series for $V$ to ones for $\mb{X}$. If both $\mb{H}$ and $\mb{G}$ have Shimura varieties, then this should just be the usual pushforward construction, as described for example in \cite{LoefflerSpherical}. The following proposition describes that construction in our framework.

\begin{prop}\label{prop:Eisenstein}
    Suppose we have a $V$-theta series
    \[
        \Theta_{\mb{H}}:C_c^\infty(V(\A^{p\infty}),\Z_p)\to\h^i_\cont(\Sh_{\mb{H}},\Z_p(j)),
    \]
    then its pushforward defines an $\mb{X}$-theta series
    \[
        \Theta_{\mb{G}}:C_c^\infty(\mb{X}(\A^{p\infty}),\Z_p)\to\h_\cont^{i+2d}(\Sh_\bG,\Z_p(j+d))/\mathrm{tors},
    \]
    where $d=\dim\Sh_\bG-\dim\Sh_{\mb{H}}$ is the codimension.
\end{prop}
\begin{proof}
    Exactly as in \cite[\S 8.2]{LSZ}, one may define a map
    \[
        \Theta_{\mb{G}}':C_c^\infty((\mb{G}\times V)(\A^{p\infty}),\Z_p)\to\h_\cont^{i+2d}(\Sh_\bG,\Q_p(j+d))
    \]
    which is left $\bG$-equivariant and right $\mb{H}$-invariant. Therefore, we have a commutative diagram
    \[
        \begin{tikzcd}
        C_c^\infty((\bG\times V)(\A^{p\infty}),\Z_p)\dar\drar["{\Theta_{\mb{G}}}"] & \\
        C_c^\infty(\mb{X}(\A^{p\infty}),\Q_p)\rar["{\Theta_{\mb{G}}'}"] & \h_\cont^{i+2d}(\Sh_\bG,\Q_p(j+d))
        \end{tikzcd}
    \]
    where the downward arrow is the $\mb{H}$-coinvariant map. One simply observes that the coinvariant map and $\Theta_{\bG}'$ both introduce the same volume factors on a basis of functions, so $\Theta_{\mb{G}}$ is integral.
\end{proof}
\begin{remark}
    The proof of \cite[Proposition 2.7]{LaiSkinner} should carry over to this setting, thereby removing the ``mod torsion'' part of the statement. One can likely axiomatize it using Loeffler's definition of Cartesian cohomology functors \cite{LoefflerSpherical, GrahamShah}.
\end{remark}

If $\mb{H}=\GL_2$ and $V$ is its standard two-dimensional representation, then Example~\ref{ex:SiegelUnits} gives a $V$-theta series as above, with $i=j=1$. This gives rise to the tame parts of the following Euler systems, with the caveat that we need the function-level statements of \cite{SakWang} in mixed characteristics. Note that all of the spherical varieties used are strongly tempered.

\subsubsection{Rankin--Selberg \cite{LLZ}}
Let $\bG=\GL_2\times\GL_2$ and $\bH=\GL_2$, embedded diagonally in $\bG$. Let
\[
    \mb{X}=\bG\times^{\bH}\mathtt{std}
\]
The pushforward construction gives rise to Beilinson--Flach elements considered in \cite{LLZ}. To obtain an Euler system, we endow the trivial line bundle over $\mb{X}$ with the $\bG$-action through the character $(g_1,g_2)\mapsto\det g_1$. Our main theorem applied to the associated $\G_m$-bundle recovers the tame norm relations of Beilinson--Flach elements \cite[Theorem 3.4.1]{LLZ}.

\subsubsection{Asai representation \cite{GrossiAF}}
Let $F$ be real quadratic field. In the previous subsection, take instead
\[
    \bG=\Res_{F/\Q}\GL_{2/F},\quad \bH=\GL_2
\]
This gives rise to the Asai--Flach classes considered in \cite{GrossiAF}. The $\G_m$-bundle defined above recovers the tame norm relation at primes $\ell$ which split in $F$.

\subsubsection{$\mathrm{GSp}_4\times\GL_2$ \cite{HJS}}
Let $\bG=\mathrm{GSp}_4\times_{\G_m}\GL_2$ and $\bH=\GL_2\times_{\G_m}\GL_2$. By $\mathtt{std}$ we now mean the 2-dimensional affine space where $\bH$ acts through its first component. The variety $\mb{X}=\bG\times^\bH\mathtt{std}$ is spherical. The trivial line bundle on $\mb{X}$ with the obvious character $\bG\to\G_m$ gives rise to the tame norm relation \cite[Proposition 8.17]{HJS}.

\subsubsection{$\mathrm{GU}(2,1)$ \cite{LSZ-U3}}
Let $E$ be an imaginary quadratic field. Consider the pair of algebraic groups defined over $\Q$:
\[
    (\bG,\bH)=(\mathrm{GU}(2,1), \GL_2\times\Res_{E/\Q}\G_m).
\]
The variety $\mb{X}=\bG\times^\bH\mathtt{std}$ is spherical by \cite[Lemma 2.5.1]{LSZ-U3}. The authors also defined a character $\mu:\bG\to\Res_{E/\Q}\G_m$ in \S 2.2 of \emph{op.~cit.} and use it to vary the class over field extensions. In our set-up, this corresponds to the trivial $\Res_{E/\Q}\G_m$-bundle over $\mb{X}$ with the action of $\mb{G}$ given by $\mu$.

An interesting additional feature is that we have a 2-variable $\Z_p$-deformation, since this is a 2-dimensional torus bundle. Let $\ell=w\bar{w}$ be a split prime, then Theorem A of \emph{op.~cit.~}includes the hypothesis that at most one of $w$ and $\bar{w}$ divides the ideal $\m$. Our result also implies the tame norm relation at $\ell$ under this hypothesis: allowing $w\bar{w}|\m$ would require a divisibility by $(\ell-1)^2$, which appears to be false in this case. On the other hand, at an inert place, the divisibility requirement is by $\ell^2-1=(\ell-1)(\ell+1)$, and this can be detected by specializing at $\ell=1$ and $\ell=-1$ separately. We hope to examine the inert case in a future paper.

\appendix
    
\section{Inverse relative Satake isomorphism for the twisted unitary Friedberg-Jacquet model}
In this appendix, we shall  study the relative spherical functions and inverse Satake transform for the twisted unitary Friedberg--Jacquet pair $(G,H)$, largely following the method of \cite{Hir10} and \cite{Off04}. Compared to the main text, we make the following changes of notations to better align with the above references.
\begin{itemize}
    \item An italic letter $G$ will denote a group variety, and $G(F)$ will denote its $F$-points.
    \item We use \emph{left} action of $G$ on the variety $X=G/H$ instead of the \emph{right} action of $G$ on $H\backslash G$ was used.
\end{itemize}
\subsection{Twisted Friedberg-Jacquet model}
Let $F$ be a $p$-adic field with {\bf odd} characteristic. Let $E/F$ be  an unramified quadratic field extension. Denote by $\overline{\cdot}$ the unique non-trivial element in $\Gal(E/F)$. Let $\CO_F$ and $\CO_E$ be the ring of integers of $F$ and $E$ respectively. Fix a purely imaginary unit $i\in E$, i.e.~$\bar{i}=-i$ such that $\CO_E=\CO_F[i]$.  Let $\varpi\in \CO_F$ be a uniformizer. Denote $\# \CO_F/\varpi$ by $q$ and normalize the norm map $|\cdot|$ on $E$ by $|\varpi|=q^{-1}$. 

Let $\left(W,\langle-,-\rangle\right)$ be a maximally split Hermitian space over $E$ of dimension $2n$. Fix a Witt decomposition $W=W_1\oplus W_2$ with $W_i$  maximal anisotropic of dimension $n$. Fix a basis $\{v_i\}$ of $W_1$ and $\{v_{-i}\}$ of $W_2$ such that $\langle v_i, v_{-j}\rangle=\delta_{ij}$. This defines a full flag  \[F_1\subset F_2\subset \cdots \subset W_1,\quad F_i=\langle v_1,\cdots, v_k\rangle.\]
Let $G$ be  the unitary group $\U(V)=\{g\in \GL(V)| (g\cdot v, g\cdot w)=(v,w)\}$ and $B\subset P$ be the  Borel subgroup and maximal parabolic subgroup stabilizing the full flag and $W_1$ respectively. Denote by $N$ the unipotent radical of $P$.

Let $\epsilon\in \GL(V)$ be the linear transformation 
$\epsilon(x+y)=x-y$ for all $x\in X$, $y\in Y$. Then $(\epsilon\cdot x, \epsilon \cdot y)=-(x,y)$ and hence $g^*= \epsilon\circ g^{-1}\circ \epsilon$
is an anti-involution on $G$.  Set  $X=\{x\in G| x^*=x\}$ 
and equip $X$ with the $G$-action
$g\cdot x:=g x g^*$.  For any $x\in\GL(V)$, let $\Phi_x$ be  its characteristic polynomial. Then $X$ is divided into disjoint $\wt{G}$-orbits $\bigsqcup_{i=0}^{2n}X_i$ where for $1\leq i\leq 2n$, $X_i=\{x\in X| \Phi_{x\epsilon}(X)=(X-1)^i(X+1)^{2n-i}\}$. Let $H\cong \Res_{E/F}\GL_{n}\subset G$ be the stabilizer of $I_{2n}\in X$. Then $H$ is 
the Levi factor of $P$.  Let $B_H\subset H$ be the subgroup of upper triangular matrices. Let $w_n$ be the anti-diagonal matrix of size $n\times n$ with anti-diagonal entities $1$. Let $ w=\begin{pmatrix} 0 & w_n \\ w_n & 0\end{pmatrix}$  and $g^H:=\bar{g}^{T}$. Then  with respect to the basis $\{v_1,\cdots, v_n, v_{-n},\cdots, v_{-1}\}$, one has 
\begin{align*}
H(F)&=\{\begin{pmatrix} a & 0 \\ 0 & w_na^{H,-1}w_n\end{pmatrix}| a\in \GL_n(E)\}\subset G(F)=\{g\in \GL_{2n}(E)| g w g^H=w\}.\\
N(F)&=\{\begin{pmatrix} I_n & b \\ 0 & I_n\end{pmatrix}| b\in\GL_n(E),\ bw_n+w_nb^H=0\}\\
P(F)&=H(F)N(F)=\{\begin{pmatrix}a & b \\ 0 & w_n a^{H,-1}w_n\end{pmatrix}|aw_nb^H=-bw_na^H,\ a\in \GL_n(E), b\in M_n(E)\}\\
T(F)&=\{\diag\{a_1,\cdots, a_n, \bar{a}_n^{-1},\cdots,\bar{a}_1^{-1}\}| a_i\in E^\times\}\subset B(F)=\{bn\in G(F)| b\in B_H(F),\ n\in N(F)\}.
\end{align*}

\subsection{Relative Cartan decomposition}
\begin{lemma}\label{F-points}One has $G\cdot I_{2n}=X_n$, and the map \[\theta:\ G\to X_n,\quad g\mapsto gg^*\] induces an isomorphism $ G/H\cong X_n$. Taking $F$-points, $ G(F)/H(F) \cong X_n(F)$.
\end{lemma}
\begin{proof}By definition, $\theta$ induces an injection $G/H\hookrightarrow X$. Note that \[X(F)=\{x\in G| x^*=x\}=\{x\in \GL_{2n}(E)| (x\epsilon)^2=I_{2n},\  w\epsilon  x \epsilon w=x^H\}\]
Let $\Gamma=\Gal(\bar{F}/F)$. Then $\Res_{E/F}\GL_{2n}(\bar{F})=\GL_{2n}(\bar{F})\times \GL_{2n}(\bar{F})$  with the   $\Gamma$-action: 
\[\sigma\cdot(g_1,g_2)=\begin{cases} (g_1^\sigma, g_2^\sigma) & \text{if}\  \sigma|_E=\mathrm{id},\\ (g_2^\sigma,g_1^\sigma) & \text{if}\ \sigma|_E=\tau.\end{cases}\]
 The involution $g\mapsto g^H$ on $G(F)$ extends to $G(\bar{F})$ by the rule $(g_1,g_2)^H=(g_2^T,g_1^T)$ and \[G(\bar{F})=\{(g_1,g_2)\in \Res_{E/F}\GL_{2n}(\bar{F})| g_2^T w g_1=w\}=\{(g, w g^{T,-1} w)| g\in \GL_{2n}(\bar{F})\}\]
 \[X(\bar{F})=\{(x, wx^{T,-1}w)| x\in \GL_{2n}(\bar{F}),\quad (x\epsilon)^2=I_{2n}\}.\]
 This forces $G\cdot I_{2n}=X_n$.  The short exact sequence of $\Gamma$-sets
\[1\to H(\bar{F})\to G(\bar{F})\xrightarrow{g\mapsto g\cdot I_{2n}} X_n(\bar{F})\to 1,\]
induce the long exact sequence of pointed sets 
\[1\to G(F)\cdot I_{2n}\to X_n(F)\to H^1(\Gamma, H)\to H^1(\Gamma, G).\]
As $\h^1(\Gamma, H)$ is trivial, we have $G(F)\cdot I_{2n}=G(F)/H(F)=X_n(F)$. 
 \end{proof}

Let $I=\{0,1\}^n$.  For $\lambda=(\lambda_1,\cdots, \lambda_n)\in I$, set \[\xi_\lambda=w\begin{pmatrix}I_n & i \varpi^\lambda w_n\\ 0 & I_n\end{pmatrix},\quad \varpi^\lambda= \diag\{\varpi^{\lambda_1},\cdots, \varpi^{\lambda_n}\}\]
 Let $\xi=\xi_{\{0,\cdots,0\}}$ and consider the $B\times H$-action on $G$ given by $(b,h)\cdot g= bg h^{-1}$. 
\begin{prop}   $B(F)\xi_\lambda H(F)$, $\lambda\in I$ are precisely all the open $(B(F)\times H(F))$-orbits in $G(F)$.  \end{prop}
\begin{proof} If there exist $g=\begin{pmatrix} b & n \\ 0 & w_n b^{H,-1} w_n\end{pmatrix}\in B(F)$ and $h=\diag\{a, w_n a^{H,-1}w_n\}\in H(F)$ such that \begin{align*}
 g=\xi_{\lambda^\prime}h\xi_\lambda^{-1}&=w \begin{pmatrix} I_n & i \varpi^{\lambda^\prime}w_n \\ 0 & I_n\end{pmatrix} \begin{pmatrix}a & 0 \\ 0 & w_n a^{H,-1}w_n\end{pmatrix} \begin{pmatrix} I_n & -i \varpi^{\lambda}w_n \\ 0 & I_n\end{pmatrix}w\\   
 &=w\begin{pmatrix} a & i (\varpi^{\lambda^\prime}a^{H,-1}w_n-a\varpi^{\lambda}w_n)\\ 0 & w_n a^{H,-1}w_n\end{pmatrix}w=\begin{pmatrix} a^{H,-1} & 0\\ i w_n(\varpi^{\lambda^\prime}a^{H,-1}-a\varpi^{\lambda}) & w_n aw_n\end{pmatrix}
\end{align*}
The equality $\varpi^{\lambda^\prime}a^{H,-1}=a\varpi^\lambda$ and the fact $a^{H,-1}$ is upper triangular forces $a$ to be diagonal and $\varpi^{\lambda^{\prime}-\lambda}=aa^H$. As $\varpi\notin N_{E/F}(E^\times)$, we must have $\lambda=\lambda^\prime$. Hence $B(F)\xi_\lambda H(F)$ are all disjoint. Actually this argument shows that $B(\bar{F})\xi_\lambda H(\bar{F})$ all coincide and are just the $\bar{F}$-points of the orbit $O=B \xi H$. One can check that $L:=\xi H\xi^{-1}\cap B\cong \U(1)^n$ and $L$ is isomorphic to the stabilizer of $\xi$. 
From the short exact sequence 
\[1\to L\to B\times H\to O\to1,\]
one deduces the long exact sequence of pointed sets 
\[1\to B(F)\xi H(F)\to O(F)\to \h^1(\Gamma, L)\to \h^1(\Gamma, B\times H)\]
Since $\h^1(\Gamma, L)\cong (\BZ/2\BZ)^n$ and $\h^1(\Gamma, B\times H)=\{e\}$, we deduce that 
\[O(F)=\sqcup_{\gamma\in \{0,1\}^n}B(F)\xi_\lambda H(F).\]

 By an easy dimension counting, one has that $O\subset G$ is open. As $G$ is irreducible, it is the unique open $B\times H$-orbit. This implies that $B(F)\xi_\lambda H(F)$ are all the open $B(F)\times H(F)$-orbit in $G(F)$.
\end{proof}
For latter use, we  describe the open $B$-orbit $O$ concretely. For $x\in X_n$, let $d_i(x)$ be the determinant of its lower left $i\times i$-block for $1\leq i\leq n$. 
Let \[X_n^\prime=\{x\in X_n| \prod_{i=1}^n d_i(x)\neq 0\}.\]
For $\lambda\in \BZ^n$, let  $\lambda^*=(-\lambda_n,\cdots,-\lambda_1)$ and  \[x_\lambda=\diag\{i\varpi^{\lambda}, -i^{-1}\varpi^{\lambda^*}\}w,\quad \varpi^\lambda=\diag\{\varpi^{\lambda_1},\cdots,\varpi^{\lambda_n}\} \]
\begin{cor} The $B$-variety $X_n^\prime$ is homogeneous. Consequently, \[X_n^\prime(F)=\theta(O)(F)=\bigsqcup_{\lambda\in I} B(F)\cdot x_\lambda.\]
\end{cor}
\begin{proof}By \cite[Lemma 5.7]{Off04}, $X_n^\prime$ is actually an open $B$-orbit (hence $\wt{B}$-orbit). Thus \[X_n^\prime(F)=\bigsqcup_{\lambda\in I} B(F)\cdot \theta(\xi_\lambda).\] Since $x_\lambda\in B(F)\cdot \theta(\xi_\lambda)$, we are done.
\end{proof}
Note $I_{2n}\in K\cdot x_0$ and for $u\in I$, \[X_u=B(F)\cdot x_u=\{x\in X_n^\prime(F)|\mathrm{val}_\varpi(d_i(x))\equiv \sum_{j=1}^iu_i \mod 2\}\]

To describe the relative Cartan decomposition for the pair $(G(F),H(F))$,  we need the generalized Cartan decomposition of the $\GL_n(E)$-module $S_n(E)=\{g\in M_n(E)| g^H=-g\}$. Note that $S_n(E)$ is a $\GL_n(E)$-module by the rule $g\cdot s=gsg^H$. Let \[\Lambda_n=\{\lambda\in\tilde{\Lambda}_n| \lambda_i\in\BZ\}\subset \tilde{\Lambda}_n=\{\lambda=(\lambda_1,\cdots,\lambda_n)| \lambda_1\geq\cdots\geq \lambda_n, \lambda_i\in\BZ\cup\{\infty\}\}\]
Here we use the convention $\infty>n$ for all $n\in\BZ$. For any $\lambda\in \tilde{\Lambda}_n$, set \[Q(\lambda)=(-\lambda_n,\cdots,-\lambda_{i+1},0,\cdots, 0)\] if $\lambda_i\geq0$ and $\lambda_{i+1}<0$. Then one has a surjection 
\[\tilde{\Lambda}_n\to \Lambda_n^+,\quad \lambda\mapsto Q(\lambda)\]
For $\lambda\in\tilde{\Lambda}_n$, set $\varpi^\lambda=\diag\{\varpi^{\lambda_1},\cdots,\varpi^{\lambda_n}\}$. Here  $\varpi^\infty=0$. 
\begin{lemma}\label{generalized Cartan}For $\lambda\in \tilde{\Lambda}_n$, set $C_\lambda=\GL_n(\CO_E)\cdot i \varpi^\lambda $. Then  $S_n(E)=\bigsqcup_{\lambda} C_\lambda$. 
\end{lemma}
\begin{proof}Let $H_n(E)=S_n(E)\cap \GL_n(E)$. Then by \cite{Jac62}, $H_n(E)=\bigsqcup_{\lambda\in \Lambda_n}C_\lambda$. For $A\notin S_n(E)-H_n(E)$, the rows $\{r_i\}$ of $A$ must satisfy an equation $\sum_{i=1}^na_ir_i=0$. Up to replacing $A$ by $sAs^H$ for suitable permutation matrix $s$, we can assume $a_1=1$ and $a_i\in\CO_E$. Take $k=\begin{pmatrix}1 & a\\ 0 & I_{n-1}\end{pmatrix}$ with $a=(a_2,\cdots,a_n)$. We have \[aAa^H=\begin{pmatrix} 0 & 0\\ 0 & A_{n-1}\end{pmatrix},\ A_{n-1}\in S_{n-1}(E)\]
Thus by induction, we have $S_n(E)=\bigcup_{\lambda} C_\lambda$.

Denote $\ell(\lambda)=\sharp\{i|\lambda_i\neq\infty\}$. Then easy to see any $k\in C_\lambda$ has rank $\ell(\lambda)$. By the Cartan decomposition above,  we can check that when  $\ell(\lambda_1)=\ell(\lambda_2)$, $C_{\lambda_1}\cap C_{\lambda_2}\neq\emptyset$ if and only if $\lambda_1=\lambda_2$. Consequently, $S_n(E)=\bigsqcup_{\lambda} C_\lambda$. 
\end{proof}

\begin{lemma}\label{Inter}For $\lambda,\mu\in \tilde{\Lambda}_n$, $Q(\mu)=Q(\lambda)$ if and only if \[\{i\varpi^\lambda-k\cdot i\varpi^u|k\in K\}\cap M_n(\CO_E)\neq \emptyset.\] 
\end{lemma}
\begin{proof} Clearly if $\lambda_n\geq0$ and $\mu_n\geq0$, then the intersection is non-zero. By the similar result for all $m\leq n$, we deduce that if $Q(\mu)=Q(\lambda)$, then 
the intersection is non-zero. For the converse direction, assume $\lambda_i\geq0$ and $\lambda_{i+1}<0$. Then one can show that if the intersection is non-empty, there exists $k\in C_{\mu}$ of the form $\begin{pmatrix} k^\prime & 0 \\ 0 & i \varpi^{\lambda_i}\end{pmatrix}$ by   applying permutation matrices and repeatedly applying the observation $\min_{a\in C_\lambda}\ell(a)=\lambda_n$ by the very definition of $C_\lambda$ and the strong inequality. Here for $a\in M_n(E)$,  $\ell(a)=\min\{\mathrm{val}_\varpi(a_{ij})|1\leq i,j\leq n\}$ ($\mathrm{val}_\varpi(0)=\infty$).  By the Cartan decomposition, we deduce that $\mu_{i}\geq0$ and $\lambda_{j}^\prime=\lambda_j$ for $j=i+1,\cdots,n$, i.e. $Q(\lambda)=Q(\mu)$. 
\end{proof}

 Consider maximal open compact subgroup $K=G(F)\cap M_{2n}(\CO_E)\subset G(F)$. Let $\Lambda_n^+=\{\lambda\in\Lambda_n|  \lambda_n\geq 0\}$. 
We have the following relative Cartan decomposition:
\begin{prop}\label{Ralative Cartan} $X_n(F)=\sqcup_{\lambda\in \Lambda_n^+}K\cdot x_\lambda$.
\end{prop}
\begin{proof}
Let \[Y_n(F)=\{y\in \GL_{2n}(E)| y^2=I_{2n},\ wyw=-y^H,\ \Phi_y(t)=(t^2-1)^n\}\] equipped with the $G(F)$-action $g\cdot y=gyg^{-1}$. Then the map \[X_n(F)\to Y_n(F),\quad x\mapsto x\epsilon\]
is a $G(F)$-equivariant bijection. Thus by Lemma \ref{F-points}, one has $G(F)\cdot \epsilon=Y_{n}(F)$. 
By the Cartan decomposition $G(F)=KP(F)$,
we have \[Y_n(F)=K\cdot \{\begin{pmatrix}I_n & b \\ 0 & -I_n\end{pmatrix}| bw_n\in  S_n(E)\}\]
Considering the action of \[\{\begin{pmatrix}a & 0\\ 0 & w_na^{H,-1}w_n\end{pmatrix}|a\in \GL_n(\CO_E)\}\subset K\cap P(F),\]
we deduce from Lemma \ref{generalized Cartan} that 
\[Y_n(F)=\bigcup_{\lambda\in \tilde{\Lambda}_n}K\cdot y_\lambda,\quad y_\lambda=\begin{pmatrix}I_n & i \varpi^\lambda w_n \\ 0 & -I_n\end{pmatrix}\]
If for $\lambda,\mu\in\Lambda_n$, $K\cdot y_\lambda=K\cdot y_\mu$. Then there exists $k=\begin{pmatrix}a & b\\ c & d\end{pmatrix}\in K$ such that 
\[\begin{pmatrix}a+i\varpi^{\lambda}w_nc & b+i\varpi^{\lambda}w_nd \\ -c & -d\end{pmatrix}=y_\lambda k=k y_{\mu}=\begin{pmatrix}a & ai\varpi^{\mu}w_n-b\\ c & ci\varpi^{\mu}w_n-d\end{pmatrix}.\]
This forces $c=0$, $a\in \GL_n(\CO_E)$ and $b\in M_n(\CO_E)$ and \[d=w_n a^{H,-1}w_n,\ aw_nb^H=-bw_na^H,\ 2bw_n a^H=i(a\varpi^\mu a^H-\varpi^\lambda )\]
By Lemma \ref{Inter}, the existence of such $k$ is equivalent to $Q(\lambda)=Q(\mu)$. Since $w\in K$, we deduce that  \[X_n(F)=\sqcup_{\lambda\in \Lambda_n^+}K\cdot z_\lambda\quad z_\lambda=
\begin{pmatrix} I_n &  0\\ i \varpi^{\lambda^*}w_n  & I_n \end{pmatrix}\in X_n(F)\]
As $x_\lambda\in K\cdot z_\lambda$, we deduce that $X_n(F)=\bigsqcup_{\lambda\in \Lambda_n^+}K\cdot x_\lambda$
\end{proof}
\subsection{Relative spherical functions}
For $x\in X_n(F)$ and $s\in\C^n$, consider the integral 
\[\omega(x;s)=\int_K|d(k\cdot x)|^{s+\eta}dk\quad |d(y)|^s=\prod_{i=1}^n|d_i(y)|^{s_i}\]
with respect to the normalized Haar measure $dk$ on $K$  with total volume $1$.  Here  \[\eta=(\eta_i)=(-1+\frac{\pi\sqrt{-1}}{\log q},\cdots,-1+\frac{\pi\sqrt{-1}}{\log q},-1/2+\frac{\pi\sqrt{-1}}{\log q})\in\C^n.\]   
By \cite[Theorem 4.1]{Off04}, this integral converges absolutely  when $\Re(s_i+\eta_i)\geq0$, $1\leq i\leq n$ and extends to a rational function of $q^{s_i}$, $i=1,...,n$. Thus by meromorphic continuation, 
$\omega(x;s)\in\CS(K\backslash X_n(F))$. 

Note that if $x\in X_n(F)$ has lower left $n\times n$-block $C$, then $\begin{pmatrix}A & N\\ 0 &  w_n A^{H,-1}w_n\end{pmatrix}\cdot x$ has lower left $n\times n$-block $ w_n A^{H,-1} w_n C A^{-1}$. Thus $d_i(p\cdot x)=\psi_i(p) d_i(x)$ for $p\in B(F)$ where \[\psi_i:\ (B\to) T\to\C^\times,\quad \mathrm{diag}\{a_1,\cdots,a_n,  \bar{a}_n^{-1},\cdots, \bar{a}_1^{-1}\}\mapsto \prod_{j=1}^iN_{E/F}(a_j)^{-1}.\] 
One can check \[|\psi(p)|^{\eta}=\prod_i|\psi_i(p)|^{\eta_i}=\delta^{1/2}(p)\]
where $\delta$  is the modulus character for the left Haar measure on $B(F)$.
Consider the Satake transform  \[\lambda_s:\ \mathcal{H}(G(F),K)\to\C(q^{s_1},\cdots, q^{s_n});\quad f\mapsto \int_{B(F)}f(p)|\psi(p)|^{-s+\eta}d_Lp=\int_{B(F)}f(p)|\psi(p)|^{-s}\delta^{1/2}(p)d_Lp\]
 Then for  $f\in\mathcal{H}(G(F),K)$,
\[f\star \omega(-;s):=\int_{G(F)}f(g)\omega(g^{-1}\cdot-)dg =\lambda_s(f)\omega(-,s)\]
Here the Haar measure $dg=dk d_Rp$ on $G(F)$ satisfies that $\mathrm{Vol}(K,dk)=1$ and $d_Lp=\delta^{-1}(p)d_R p$.

Note that the Weyl group $W=W(G,B)=N_G(T)/T=S_n\ltimes(\BZ/2\BZ)^n$ is   generated by $S_n$ and $\tau$ where $S_n$ acts on $t=(a_1,\cdots, a_n,   \bar{a}_n^{-1},\cdots,  \bar{a}_1^{-1})$ by permuting indices and $\tau$ fixes $a_i$, $1\leq i<n$ and $\tau(a_n)= \bar{a}_n^{-1}$.  The induced action on $\psi^s$ gives an action of $W$ on $s$. Introduce new variables $z_i=-\sum_{j=i}^ns_j$ for $1\leq i\leq n$. Then $\psi^z:=\psi^s=\prod_{i=1}^nN(a_i)^{z_i}$. The $W$-action on $\psi^s$ induces an action of $W$  on $z$: $S_n$ acts by permutation of indices and \[\tau(z_1,\cdots,z_n)=(z_1,\cdots,-z_n)\]

Let $\Sigma$ be the root system of type $C_n$ with positive roots $\Sigma^+=\Sigma_s^+\sqcup \Sigma_\ell^+$ where 
\[ \Sigma_s^+=\{e_i-e_j,e_i+e_j|1\leq i<j\leq n\},\quad \Sigma_\ell^+=\{2e_i|1\leq i\leq n\}\]
Let $R$ be the root system of type $BC_n$ with positive roots  $R^+=R_s^+\sqcup R_\ell^+$ where  $R_\ell^+=\{2e_i|1\leq i\leq n\}$,
\[ R_s^+=R_{s,1}^+\sqcup R_{s,2}^+\quad R_{s,1}^+=\{e_i-e_j,e_i+e_j|1\leq i<j\leq n\},\ R_{s,2}^+=\{e_i|1\leq i\leq n\}. \]
Here $e_i\in\BZ^n$ is the $i$-th unit vector. View $\Sigma, R\subset \C^n$ and consider the $W$-equivariant pairing
\[\langle-,-\rangle:\ \BZ^n\times\C^n\to\C,\quad (t,z)\mapsto\sum_it_iz_i\]
\begin{theorem}\label{FE} Let $\omega(x;z):=\omega(x;s)$ and $\Omega(x;z):=\frac{\omega(x;s)}{\omega(I_{2n};s)}$. The function $G(z)\omega(x,z)$ is holomorphic on $\C^n$ and $W$-invariant.  Here \[G(z)=\prod_{\alpha\in R_{s,1}^+}\frac{1+q^{\langle\alpha,z\rangle}}{1-q^{\langle\alpha,z\rangle-1}}\prod_{\alpha\in R_{s,2}^+}\frac{1+q^{\langle\alpha,z\rangle-1/2}}{1-q^{\langle\alpha,z\rangle-1/2}}\]
\end{theorem}
\begin{proof}First we consider the case $n=1$. Consider the similitudes unitary group \[\wt{G}=\{g\in \GL_2(E)| gwg^H=\lambda(g)w\}= J (\GL_2(F)\times_{F^\times} E^\times) J^{-1},\quad J=\begin{pmatrix} i & 0\\ 0 & 1\end{pmatrix} \]
Then $\wt{G}$ also acts on $X_1$ by the rule: $g\cdot x= gx\epsilon^{-1}g \epsilon$. Let $\wt{H}$ be the stabilizer of $I_2$. Let \[Y_1=J^{-1}X_1 J=\{\begin{pmatrix}a & b\\ c & a\end{pmatrix}\in \GL_2(F)| a^2-bc=1\}\]
and consider the $\GL_2(F)$-action on $Y_1$ given by: $g\cdot y=gy \epsilon g^{-1}\epsilon$. Then 
conjugation by $J$ induces an isomorphism \[\wt{G}(F)/\wt{H}(F)=X_1(F)\cong Y_1(F)=\GL_2(F)/T(F)\]
where $T\subset \GL_2$ is the diagonal torus.
Let $\wt{K}=\wt{G}(F)\cap \GL_2(\CO_E)= K\times\diag\{1,\CO^\times\}$. Then \[\omega(x;z)=\int_K|d(k\cdot x)|^{s+\eta}dk=\int_{\wt{K}}d(k\cdot x)|^{s+\eta}dk=\int_{\GL_2(\CO)}d(k\cdot y)|^{s+\eta}dk,\quad y=J^{-1}xJ\] 
With all these preparation, we deduce  from \cite[Prop 4.3, Lemma 5.2, Lemma 5.6, Lemma 5.13]{Off04} that 
\begin{itemize}
    \item there exists an absolute constant $C$ such that $\omega(I_2;z)=C\frac{1-q^{-1/2}q^z}{1+q^{-1/2}q^z}$,
    \item $\Omega(x;z)\in \C[q^z,q^{-z}]^W$.
\end{itemize} 
In other words, $\frac{1+q^{-1/2}q^z}{1-q^{-1/2}q^z} \omega(x;z)$ is fixed by $W$.

Now  we can adapt the induction argument in \cite[Thm 2.3]{HK14} to show $\frac{1+q^{-1/2}q^{z_n}}{1-q^{-1/2}q^{z_n}} \omega(x,z)$ is $\tau$-invariant when $n\geq2$. Consider the element \[w_\tau=\begin{pmatrix} I_{n-1} & 0 & 0
\\ 0 & w_2 & 0\\ 0 & 0 & I_{n-1}\end{pmatrix}\]
and the attached parabolic subgroup $P_\tau=B\cup B w_\tau B$. For any $p=(p_{i,j})\in P_\tau$, let \[\rho(p)=\begin{pmatrix} p_{n,n} & p_{n,n+1}\\ p_{n+1,n} & p_{n+1,n+1}\end{pmatrix}\in \U(2)=\{g\in \GL_2(E)| gg^H=I_2\}\]
Consider the $P_\tau\times E^\times$ action on $X_n(F)\times E^2$ (the $E^2$-component is realized as column vectors):
\[(p,r)\ast (x,v)=(p\cdot x, \rho(p) v r^{-1})\]
For $x\in X_n(F)$,$v=[v_1, v_2]^T\in E^2$, let  \[g(x,v)=\det \left[\begin{pmatrix} -v_2 & v_1 & 0\\ 0 & 0 & I_{n-1}\end{pmatrix}\cdot x_{n+1}\right]\]
Here $x_{n+1}$ is the lower right $(n+1)\times (n+1)$-block of $x\epsilon w$ and $p\cdot x=pxp^H$. Then one can check that 
$g(x,v_0)=(-1)^{\lfloor \frac{n}{2}\rfloor}d_n(x)$ for $v_0=[1\ 0]^T$ and by \cite[Lemma 2.5(i)]{HK14}, \[g((p,r)\ast (x,v))=\psi_{n-1}(p)N(r)^{-1}g(x,v)\quad \forall\ (p,r)\in P_\tau\times E^\times\]
Moreover similarly to \cite[Lemma 2.5(ii)]{HK14}, there exists $D_x\in S_2(E)$ such that $g(x,v)=v^HD_x v$ for all $v\in E^2$. One has
$(-1)^{\lfloor \frac{n}{2}\rfloor}d_{n-1}(x)^{-1}D_x w_2\epsilon_2\in X_1(F)$ for $\epsilon_2=\diag\{1,-1\}$.  Then by the argument of \cite[Theorem 2.3]{HK14}, we have $\frac{1+q^{1/2}q^{z_n}}{1-q^{-1/2}q^{z_n}} \omega(x,z)$ fixed by $\tau$.

Now we adapt the argument in \cite[Thm 2.1]{HK14} to show $G_1(z)\omega(x,z)$ is $S_n$-invariant, where
 \[G_1(z)=\prod_{1\leq i<j\leq n}\frac{1+q^{z_i-z_j}}{1-q^{z_i-z_j-1}}.\]
 Consider the embedding \[K_0=\GL_n(\CO_F)\to K,\quad h\mapsto \begin{pmatrix} w_n h^{H,-1}w_n & 0\\ 0 & h\end{pmatrix}\]
and the normalized Haar measure $dh$ on $K_0$. 
One has \[\omega(x;z)=\int_{K_0} \int_K |d(hk\cdot x)|^{s+\eta}dkdh=\int_K \xi_*^h(L(k\cdot x)w_n;z)dk\] 
Here $L(k\cdot x)$ stands for lower left $n\times n$ block of $k\cdot x$ and $\xi_*^h$ is a spherical function on $S_n(E)$ defined by 
\[\xi_*^h(y;z)=\int_{K_0}|\tilde{d}(h\cdot y)|^{s+\eta}dh,\quad h\cdot y=hyh^H.\]
Here $|\tilde{d}(x)|^s=\prod_i|\tilde{d}_i(x)|^{s_i}$ where $\tilde{d}_i$ is the determinant of the lower right $i\times i$-block. Then the desired invariance follows from \cite[Proposition 2.2]{HK14}.

With all these results and the fact $W$ is generated by $S_n$ and $\tau$, we can simply check that $G(z)\omega(x,z)$ is  $W$-invariant. Then proceeding as \cite[Thm 2.6]{HK14}, we have  $G(z)\omega(x,z)$ is holomorphic.

\end{proof}
To determine $\omega(x,z)$, we follow Casselman--Shalika--Hironaka. Consider the Iwahori subgroup \[U=\{\nu=(u_{ij})\in K| u_{ii}\in\CO_E^\times\ \forall\ 1\leq i\leq 2n, u_{ij}\in\varpi\CO_E\ \text{if}\
 i>j\}\]
From \cite[Lemma 5.8]{Off04} (applied to $\GL_{2n}(E)$), we have
\begin{lemma}\label{Iwahori} For any $\lambda\in\Lambda_n^+$ and $x\in U$
\[|d_i(x\cdot x_\lambda)|=|d_i(x_\lambda)|=q^{\lambda_1+\cdots+\lambda_i}  \]
\end{lemma}
For $u\in I\cong (\BZ/2\BZ)^n$, let $ c_\lambda=(-1)^{\sum_i \lambda_i(n-i+1)}q^{-\sum_i \lambda_i(n-i+1/2)}$ and
\[\delta_u(x_\lambda;z)=\begin{cases} |d(x_\lambda)|^{s+\eta}=c_\lambda q^{-\langle \lambda,z\rangle} & \text{if}\ x_\lambda\in X_u;\\
0 & \text{otherwise}.\end{cases}\]
Consider the refined spherical function
\[\omega_u(x;z):=\omega_u(x;s)=\int_K|d(k\cdot x)|_u^{s+\eta}dk,\quad  |d(y)|_u^s=\begin{cases} \prod_{i=1}^n|d_i(y)|^{s_i}\ &\text{if}\ y\in X_u;\\
0 & \text{otherwise}\end{cases}\]
Then for each character $\chi:\ I\to\C^\times$, we have 
\[\sum_{u\in I}\chi(u)\omega_u(x;z)=\omega(x;z_\chi)\]
where $z_\chi$ is obtained by adding $\frac{\pi\sqrt{-1}}{\log q}$ to $z_i$ for suitable $i$ (according to $\chi$). Then by Theorem \ref{FE}, one can take a suitable character $\sigma(\chi)$ for each $\sigma\in W$ such that
\[\omega(x;z_\chi)=\Gamma_\sigma(z_\chi)\omega(x;\sigma(z_\chi))=\Gamma_\sigma(z_\chi)\omega(x;\sigma(z)_{\sigma(\chi)})\]
where $\Gamma_\sigma=G(\sigma(z))/G(z)$. This gives vector-wise functional equations
\[(\omega_u(x;z))_{u\in I}=A^{-1}\cdot G(\sigma,z)\cdot \sigma{A}\cdot (\omega_u(x;\sigma(z))_{u\in I} \quad \sigma\in W\]
 where $A=(\chi(u))_{\chi, u}$, $\sigma{A}=(\sigma(\chi)(u))_{\chi,u}$
 and $G(\sigma,z)$ is a diagonal matrix of size $2^n\times 2^n$
whose $(\chi,\chi)$-component is $\Gamma_\sigma(z_\chi)$.  Note that 
\[Q=\sum_{\sigma\in W}[U\sigma U: U]^{-1}=\frac{\omega_{2n}(-q^{-1})}{(1-q^{-2})^n},\quad \omega_m(t)=\prod_{i=1}^m(1-t^i)\]
Let $\gamma(z)=\prod_{\alpha\in\Sigma_s^+}\frac{1-q^{2\langle\alpha,z\rangle-2}}{1-q^{2\langle\alpha,z\rangle}}\prod_{\alpha\in\Sigma_\ell^+}\frac{1-q^{\langle\alpha,z\rangle-1}}{1-q^{\langle\alpha,z\rangle}}$ and \[c(z)=\gamma(z)G(z)=\prod_{\alpha\in R_{s,1}^+}\frac{1+q^{\langle\alpha,z\rangle-1}}{1-q^{\langle\alpha,z\rangle}}\prod_{\alpha\in R_{s,2}^+}\frac{1+q^{\langle \alpha,z\rangle-1/2}}{1-q^{\langle\alpha,z\rangle-1/2}} \prod_{\alpha\in R_\ell^+}\frac{1-q^{\langle\alpha,z\rangle-1}}{1-q^{\langle\alpha,z\rangle}}\]

\begin{theorem}\label{FE1} For any $\lambda\in\Lambda_n^+$, \begin{align*}
    \omega(x_\lambda,z)&=\frac{c_\lambda(1-q^{-2})^n}{\omega_{2n}(-q^{-1})}\cdot\sum_{\sigma\in W}\gamma(\sigma(z))\Gamma_\sigma(z)q^{-\langle \lambda,\sigma(z)\rangle}\\
    &=\frac{c_\lambda(1-q^{-2})^n}{\omega_{2n}(-q^{-1})}G(z)^{-1}Q_\lambda(z),\quad Q_\lambda(z):=\sum_{\sigma\in W}\sigma(q^{-\langle\lambda,z\rangle}c(z)) 
\end{align*}
\end{theorem}
\begin{proof}This follows from \cite[Thm 2.6]{Hir10}. One can adapt  \cite[Section 3.2]{HK14} to show our $(G, X)$ satisfies the assumption A1-A4 in \emph{loc.~ cit.} (apply the  induction argument there to $X_n(F)\epsilon w$).
\end{proof}

 Assigning the parameters $\{t_\alpha| \alpha\in R\}$ take the following values:
\begin{itemize}
\item if $\alpha \in \Sigma$ is long, then $t_{\alpha/2}=1$ and $t_{\alpha}^{1/2}=q^{-1/2}$;
\item if $\alpha \in \Sigma$ is short, then $t_\alpha=-q^{-1}$;
\item if $\alpha\notin R$, then $t_\alpha^{1/2}=1$.
\end{itemize}
Let \[ P_z(\lambda):=P_\lambda(q^z)=V_\lambda^{-1}\sum_{\sigma\in W}\sigma(q^{-\langle \lambda,z\rangle}c(z))\quad c(z)=\prod_{\alpha\in R^+}\frac{1-t_\alpha t_{2\alpha}^{1/2} q^{\langle\alpha,z\rangle}}{1- t_{2\alpha}^{1/2} q^{\langle\alpha,z\rangle}}\] be the Macdonald polynomial associated to the pair $(R,\frac{\Sigma}{2})$. Then $\Omega(x_\lambda;z)=c_\lambda V_\lambda V_0^{-1} P_z(\lambda)$. Let \[\Delta(z)=\prod_{\alpha\in R}\frac{1-t_{2\alpha}^{1/2}q^{\langle \alpha,z\rangle}}{1-t_\alpha t_{2\alpha}^{1/2}q^{\langle \alpha,z\rangle}}\]
Let $D$ be the direct product of $n$ copies of $\sqrt{-1}(\R/\frac{2\pi}{\log q}\BZ)$ and $dz$ be the Lebesgue measure on $D$ of totally volume one. Let $d_\mu z=\frac{1}{\# W}V_0\Delta(z)dz$. 
\begin{theorem}\label{RSUFJ}The relative Satake transform 
\[ \CS(K\backslash X_n(F))\to \C[q^{\pm z_1},\cdots, q^{\pm z_n}]^W,\quad \varphi\mapsto \hat{\varphi}(z):=\int_{X_n(F)}\varphi(x)\Omega(x;z)dx\]
is an $\sh{H}(G(F),K)$-module isomorphism. Moreover, one has the relative inverse Satake transform: 
\[\varphi(x)=\int_D \hat{\varphi}(z)\Omega(x;z)d_\mu(z)\quad \forall x\in X_n(F)\]
In particular, $\CS(K\backslash X_n(F))$ is free of rank $2^n$ over $\mathcal{H}(G(F),K)$.
\end{theorem}

\bibliographystyle{alpha}
\bibliography{Ref}

\end{document}